\documentclass{article}
\usepackage[utf8]{inputenc}
\usepackage[greek,english]{babel}
\usepackage{alphabeta}
\usepackage{amsfonts}
\usepackage{graphicx}
\usepackage{amssymb}
\usepackage{parskip}
\usepackage{amsthm}
\usepackage{geometry}
\usepackage{array}
\usepackage{amsmath}

\theoremstyle{theorem}
\newtheorem{theorem}{Theorem}[section]
\newtheorem{corollary}{Corollary}[theorem]
\newtheorem{lemma}[theorem]{Lemma}
\newtheorem{proposition}{Proposition}[section]
\theoremstyle{definition}
\newtheorem{definition}{Definition}[section]
\newtheorem{remark}{Remark}

\title{Framed Braid Equivalences}
\author{Anastasios Kokkinakis}
\date{}

\begin{document}

\setlength{\parindent}{20pt}

\maketitle

\begin{abstract}
    We introduce framed versions of the $L$-moves and prove a one move theorem for the extension of the Markov theorem for framed braids. We further introduce framed versions of the Hilden and Pure Hilden groups, we give presentations and we use them to state and prove a framed version of the Birman theorem for framed links in plat representation.
\end{abstract}

\section{Introduction}
\label{sec:intro}
Framed knots and links are an extension of knots and links that one can visualize as closed loops of knotted and linked flat ribbons. They offer a very useful presentation of closed, connected  and orientable (c.c.o.) 3-manifolds due to the Lickorish-Wallace theorem, stating that every such manifold can be obtained by Dehn surgery on a framed link. The Kirby calculus provides a set of moves on framed link diagrams that present homeomorphic c.c.o. 3-manifolds. Framed braids were introduced in 1992 by Smolinski and Ko in order to describe framed links as closures of framed braids and provide the framed braid theoretic analogue of Kirby calculus.  \\
\indent  As a result, and in a manner similar to classical links and braids, it was natural to consider first framed versions of the classical Alexander and Markov theorems. The Alexander  theorem 
states that every oriented knot or link can be isotoped to the standard closure of some braid, while the Markov theorem provides the equivalence relation between braids that correspond to isotopic knots or links. This paper begins with exploring variations of these theorems, such as a framed version of the $L$-move theorem by Lambropoulou and Rourke, a one-move theorem refining the Markov theorem. \\
\indent In a furhter development, in 1975 Hilden introduced a family of subgroups of the classical braid groups with even number of strands, now known as the Hilden groups, whose elements have the property that their plat closures are isotopic to the plat closures of the corresponding identity braids. It is easily seen that every  knot or link can be isotoped to the plat closure of some braid. Subsequently, in 1976 Birman, using the Hilden groups, proved an analogue of the Markov theorem for isotopic links in plat representation. \\ 
\indent In this paper we continue with exploring the plat closure of framed braids. We define the framed Hilden groups and give a group presentation analogous to Tawn's presentation for the Hilden groups. We further define the pure framed Hilden groups for which we also provide a presentation. Finally, we state and prove a framed version of the Birman theorem concerning plat representations of  framed links.\\
\indent The paper is organized as follows. In Sections \ref{sec:preliminaries} and \ref{sec:stand_closure} we recall the main definitions and results on classical links and braids and the theory connecting the two notions via the standard closure of braids. We describe the $L$-moves as introduced in \cite{lambropoulou1997markov} and give an equivalent definition. In Sections \ref{sec:framed} and \ref{sec:framed_closure} we present definitions and discuss the various presentations (blackboard or integer) for framed links and framed braids. In Section \ref{sec:framed_closure} we introduce the framed version of the $L$-moves on braid diagrams and define the framed $L$-equivalence for framed braid diagrams and framed links in blackboard framing. In Section \ref{sec:Hilden_plat} we recall the theory of links in plat representation. We recall presentations for the Hilden groups and discuss the Birman theorem. In Section \ref{sec:framedHilden_plat} we define the framed Hilden and pure framed Hilden groups and give presentations with generators and relations. Finally, in Section \ref{sec:framed_birman} we present a framed version of the Birman theorem.

\section{Preliminaries: Classical Braids and Classical Links}
\label{sec:preliminaries}

\subsection{Classical Braids}

Consider in $\mathbb{R}^3$, or its compactification $S^3$, the sets of points $\{A_i=(i,0,0)|i=1,\ldots ,n\}$ and $\{B_j=(j,0,1)|j=1,\ldots ,n\}$. A \textit{strand} connecting a $B_j$ to an $A_i$ is a simple non-intersecting arc that when traversed from $B_j$ to $A_i$ the $z$-coordinate decreases monotonically. In other words, each horizontal plane intersects with the strand in exactly one point. The points $A_i$ and $B_j$ shall respectively be called the \textit{lower} and \textit{upper} endpoints of the strand. A \textit{geometric braid on n strands} is defined as a collection of pairwise non-intersecting $n$ strands joining $B_1,\ldots ,B_n$ to $A_1,\ldots ,A_n$ in any order.\\
\indent Two geometric braids on $n$ strands are called \textit{isotopic} if one can be continuously deformed into the other in the class of geometric braids. Isotopy between two geometric braids is an equivalence relation. The equivalence classes are called \textit{braids on $n$ strands}. We shall use the term \textit{(geometric) braid} to mean an equivalence class of braids and a concrete representative of such a class.\\
\indent Given two geometric braids $b_1, b_2$ one can obtain their product by putting them end to end. More precisely, we contract $b_1$ vertically to half its height while keeping the upper endpoints fixed, and we also contract $b_2$ vertically to half its height, only this time we keep the lower endpoints fixed. The lower endpoints of $b_1$ will now be of the form $(i,0,1/2)$, same as the upper endpoints of $b_2$. The union of the contracted braids is defined as the \textit{product of $b_1$ and $b_2$} and is denoted $b_1b_2$. The equivalence class of this geometric braid is defined as the product of the equivalence classes of $b_1$ and $b_2$. This operation turns the set of braids (as equivalence classes) on $n$ strands into a group, with the inverse of a braid $a$ being the mirror image of the (geometric) braid in the plane $\{z=1/2\}$ and the unit element being the braid with $n$ vertical strands connecting $B_i$ to $A_i$ for $i=1,\ldots ,n$. The group of equivalence classes of braids on $n$ strands is called the braid group $B_n$.

\noindent Given a braid $b$ there is a representative diagram that is a projection in $\mathbb{R} \times \{0\} \times [0,1]$ with the following properties: 

\begin{enumerate}
    \item The projections of the strands are not tangent to each other.
    \item No point in $\mathbb{R} \times \{0\} \times [0,1]$ is the projection of three or more points from different strands.
    \item Double points that have the same projection in $\mathbb{R} \times \{0\} \times [0,1]$ occur in different $z$-coordinates.
\end{enumerate}

\noindent This projection is called a \textit{braid diagram} of $b$. Two braid diagrams representing braids on $n$ strands are said to be \textit{equivalent} if they are connected by a finite sequence of plane isotopies that preserve the braid diagram structure and the braid isotopy moves between depicted in Fig.~\ref{fig:BraidIsotopy}. Again, like in geometric braids, we shall refer to the equivalence classes and a representative from this class as a braid diagram. Braid diagrams can be endowed with a natural orientation on the strands simply by making the strands point downwards (or upwards). Lastly we note that two braids are isotopic if and only if any two of their diagrams are equivalent.\\
\indent Since geometric braids form a group this is also true for braid diagrams. So, naturally, a group is defined on generators and relations between the generators. The classical generators in $B_n$ are denoted $σ_i$ for $i=1,\ldots ,n-1$ (see Fig. \ref{fig:braidisomorphism}).

\begin{figure}[htp]
    \centering
    \includegraphics[width=9cm]{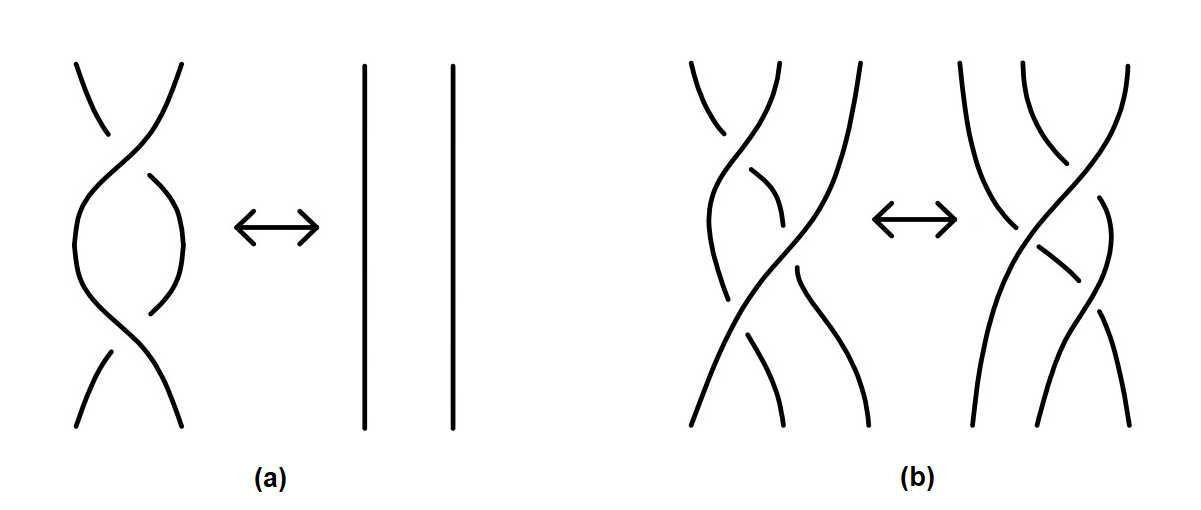}
    \caption{Braid isotopy moves.}
    \label{fig:BraidIsotopy}
\end{figure}

\begin{theorem}[Artin]
\label{Artin}
The group $B_n$ is characterized algebraically to be the group with presentation:
\[ 
B_n = 
    \Biggl\langle 
       \begin{array}{l|ccl}
            σ_1,\ldots ,σ_{n-1} & σ_iσ_jσ_i = σ_jσ_iσ_j  &  \forall \, i,j \quad \textrm{so that} \quad |i-j|=1  \\
                            & σ_iσ_j = σ_jσ_i   &  \forall \, i,j \quad \textrm{so that} \quad |i-j| \ge 2
        \end{array}
     \Biggr\rangle
     \]
\end{theorem}

\begin{figure}[htp]
    \centering
    \includegraphics[width=5cm]{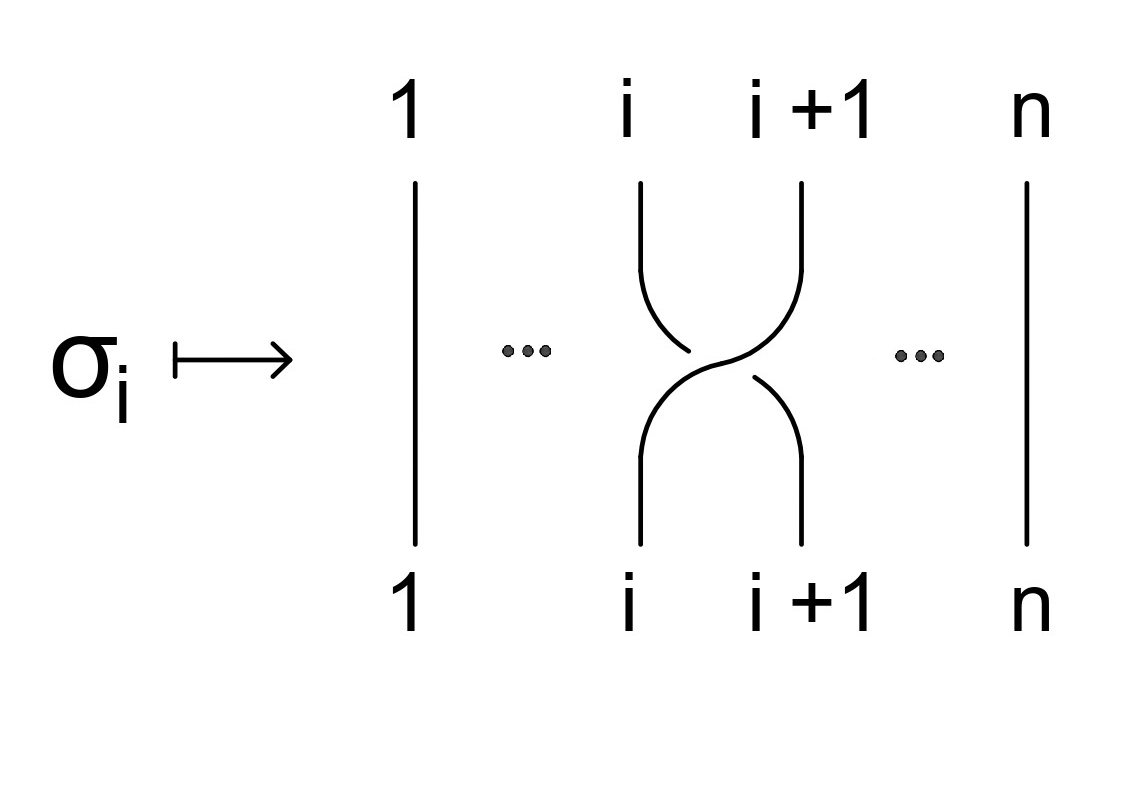}
    \caption{The Artin generators.}
    \label{fig:braidisomorphism}
\end{figure}

\noindent Theorem \ref{Artin} allows us to look at braids both geometrically and algebraically.

\subsection{Classical Knots and Links}

By a \textit{knot} we mean an embedding of $S^1$ into $S^3$ (or $\mathbb{R}^3$). Two knots $K_1, K_2$ are said to be \textit{isotopic} if there is an orientation preserving homeomorphism $f:S^3 \to S^3$ so that $f(K_1)=K_2$, while a \textit{link} in $S^3$ of $c$ components is an embedding of $c$ copies of $S^1$. Of course a knot is a link with one component. 
 A \textit{knot diagram} $D$ of a knot $K$ is a projection of $K$ into $S^2$ (or $\mathbb{R}^2$) so that at most two points are projected to the same point in $S^2$. This projection is also endowed with over/under information on the double points. Two knot diagrams $D_1,D_2$ represent isotopic knots if and only if they are connected by a finite sequence of the Reidemeister I, II and III moves and plane isotopies (see Fig.~\ref{fig:reidemeister}).

\begin{figure}[htp]
    \centering
    \includegraphics[width=11cm]{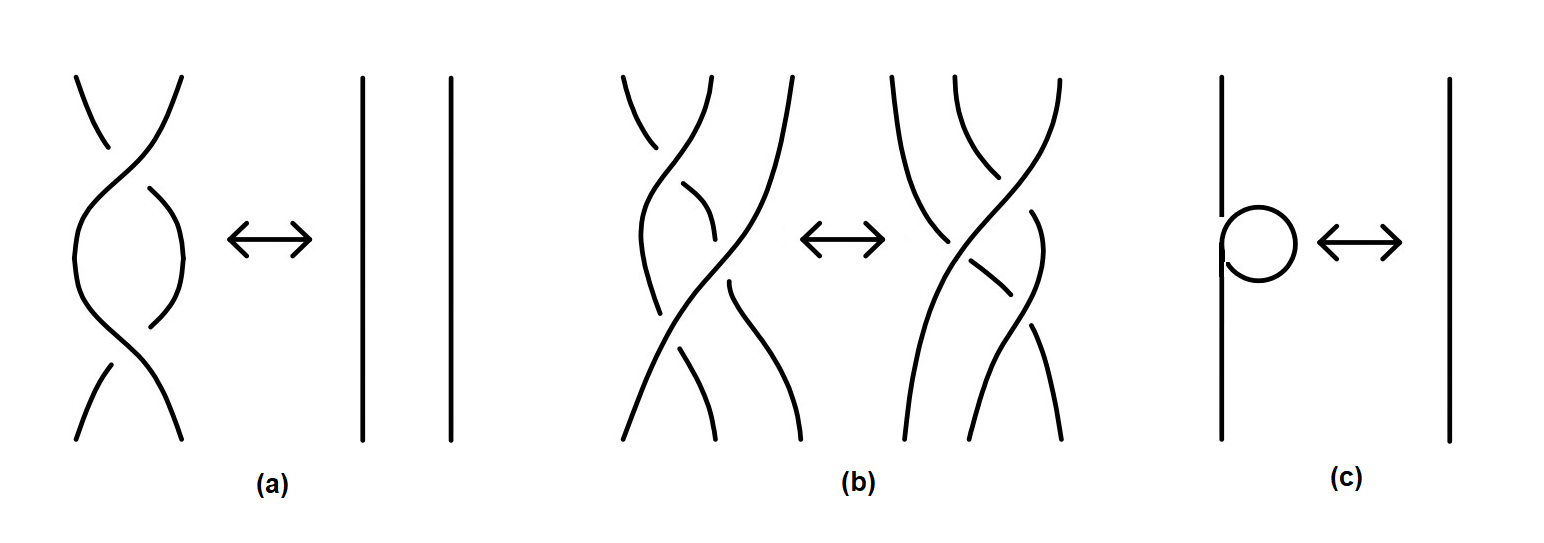}
    \caption{The Reidemeister moves.}
    \label{fig:reidemeister}
\end{figure}

Given a braid $b$ and a braid diagram $D$ of $b$ one can produce an oriented link diagram from $D$ by connecting the corresponding upper and lower endpoints of the braid diagram with simple arcs, as abstracted in Fig.~\ref{fig:closure}. This link diagram is called the (standard) \textit{closure} of $D$ and is denoted $\hat{D}$, see Fig.~\ref{fig:closure}. The isotopy class of $\hat{D}$ shall be called the \textit{closure of $b$}.

\begin{figure}[htp]
    \centering
    \includegraphics[width=8cm]{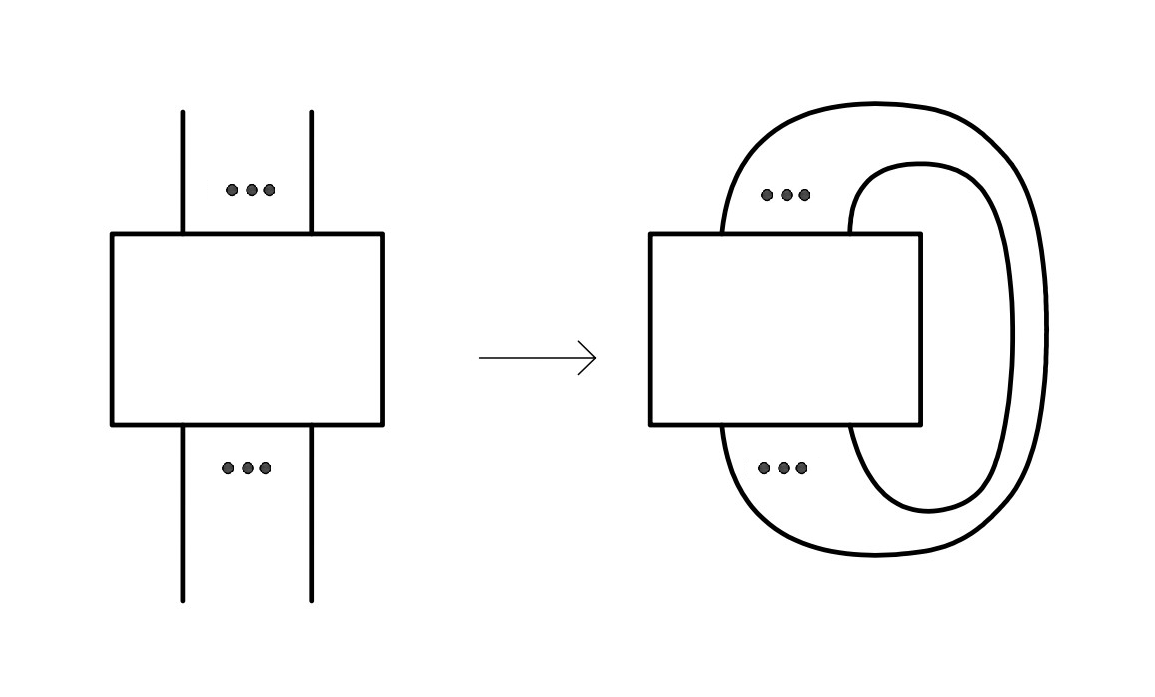}
    \caption{The standard closure of a braid.}
    \label{fig:closure}
\end{figure}

\section{Framed links and framed braids}
\label{sec:framed}

A framed knot can be described in various ways. We can picture framed knots as a closed knotted, twisted orientable band or ribbon, or we can picture it as a non-null homologous simple curve on the boundary surface of a solid torus or simply as a classical knot with extra information. We present these various depictions of framed knots (and links) and how one is connected to the other.

\subsection{Representations of framed links}

A \textit{solid torus}, $V$, is a space homeomorphic to $S^1 \times D^2$, i.e. $V=h(S^1 \times D^2)$ for some homeomorphism $h$. The curve $h(S^1 \times 0)$ is called the \textit{core of $V$} (see $K$ in fig \ref{fig:framed tref}). A \textit{meridian} of $V$ is a non-contractible, simple, closed curve on $\partial V$ that bounds a disc in $V$. A \textit{longitude} of $V$ is a non-contractible, simple closed curve on $\partial V$ that intersects transversely some meridian of $V$ in a single point.\\
\indent It is known that the fundamental group of the boundary $\partial V$ of the solid torus $V$ is the free abelian group on two generators $\mathbb{Z} \times \mathbb{Z}$. Therefore one can define a longitude for $V$ as a simple, closed curve on $\partial V$ that is of the form $(k,1)$ in the fundamental group of the boundary $\partial V$, $k \in \mathbb{Z}$. A meridian is a simple, closed curve of the form $(1,0)$.

\begin{definition} Let $D$ be a diagram of an oriented link $L$ and let $c_1, c_2$ be two components of $D$. The \textit{linking number} between $c_1$ and $c_2$ is the total number of positive crossings minus the total number of negative crossings, divided by $2$,  over all crossings between $c_1$ and $c_2$. A \textit{positive} crossing is when one has to rotate the over arc counterclockwise, in order to align the two arcs involved, so that they both  point to the same direction. A \textit{negative} crossing is when the upper strand has to rotate clockwise. The linking number counts the algebraic number of times one component winds around the other one. Further, it is an ambient isotopy invariant as it respects Reidemeister II and III moves and it is not affected by the Reidemeister I move. Thus we talk about the \textit{linking number of two link components} as the linking number between them in any diagram representative of the link.    
\end{definition}

\begin{figure}[htp]
    \centering
    \includegraphics[width=5cm]{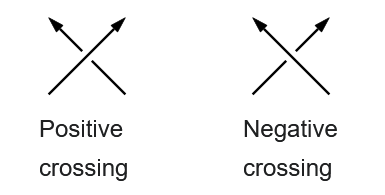}
    \caption{The two types of crossing in a link diagram.}
    \label{fig:crossings}
\end{figure}

\begin{definition}[Framing]
Let $K$ be an oriented knot in $S^3$ (or $\mathbb{R}^3$). Then $K$ has a regular neighbourhood in $S^3$ that is a solid torus $S^1 \times D^2 = V$, where $K$ is the core of $V$. Let $K_0$ be a longitude in $V$ with the same orientation as $K$. The linking number $lk(K,K_0)$ between $K$ and $K_0$, is called a \textit{framing} of $K$. For an example view Fig.~\ref{fig:framed tref}a. A \textit{framed link} of $c$-components is a link of whom each one of the $c$ components is a framed knot. To a framed link of $c$ components one can assign a framing vector $\vec{λ}=(λ_1,\ldots ,λ_c)$, where $λ_i \in \mathbb{Z}$ is the framing of the $i$-th component.
\end{definition}

\begin{figure}[htp]
    \centering
    \includegraphics[width=10cm]{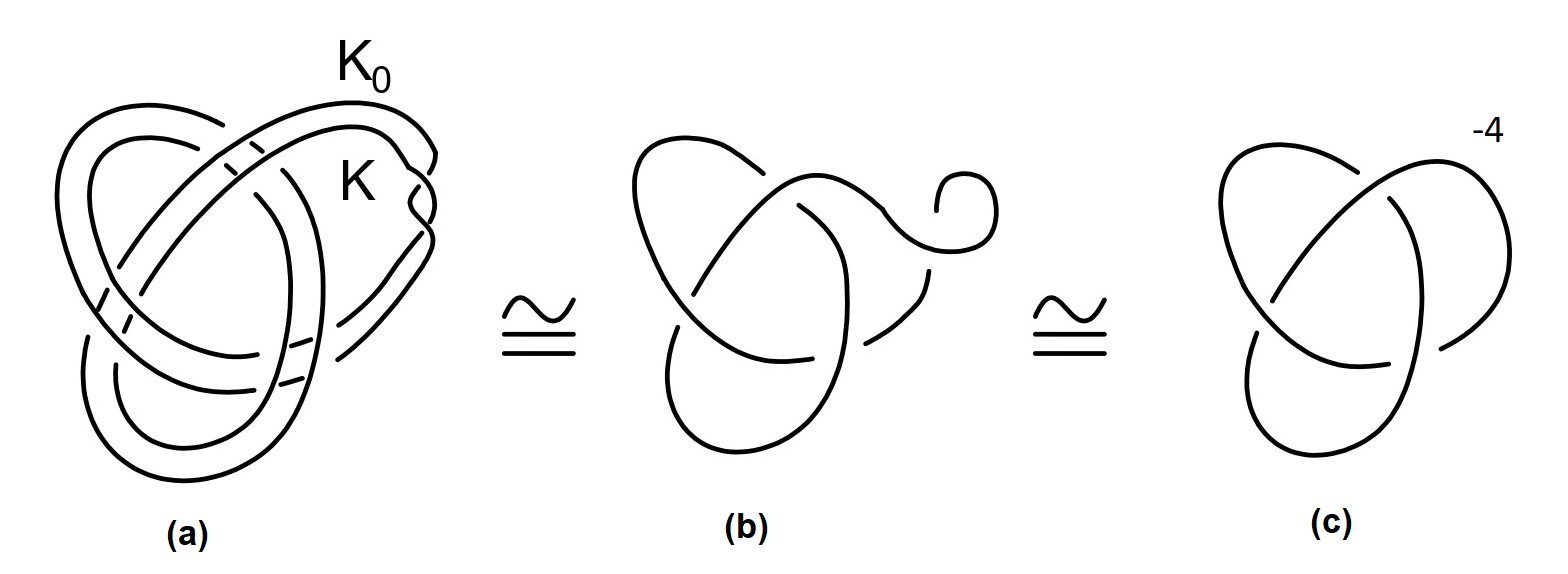}
    \caption{Different framing representations of a trefoil with framing -4.}
    \label{fig:framed tref}
\end{figure} 

We note that a framing for a knot can be any integer, as it also depends on how many times the longitude follows a meridian of $V$, which is the first coordinate in the fundamental group $\mathbb{Z} \times \mathbb{Z}$ of the torus $\partial V$.

\noindent Another standard way of representing the framing of a link is as follows:
\begin{definition}
Given an oriented knot diagram $D$, one can always assign a framing by taking the writhe of $D$ as the framing. This is called the \textit{blackboard framing} of $D$. By \textit{writhe} we mean the total number of positive crossings minus the total number of negative crossings of $D$. For an example view Fig.~\ref{fig:framed tref}b.
\end{definition}

\noindent  We observe that the blackboard framing of $D$ can change by adding to it a number of positive or negative curls. It is easy to confirm that the two definitions of framing are equivalent. Indeed, given an oriented knot $K$, a diagram $D$ of $K$ and a framing $λ$ of $K$, one can always draw a diagram of the corresponding longitude $K_0$ by the following steps: 

\begin{enumerate}
    \item Draw another curve $D_0$ $ε$-close and parallel to $D$, and following the same over/under pattern of $D$.
    \item The linking number between $D$ and $D_0$ is equal to the writhe $w(D)$ of $D$. If $λ$ is different than $w(D)$, let us say $w(D)<λ$, then we link the two diagrams on an arc of $D$ with no crossings, by adding $2(λ-w(D))$ positive consecutive crossings between $D$ and $D_0$ (see Fig.~\ref{fig:framed tref}a).
\end{enumerate}

\begin{figure}[htp]
    \centering
    \includegraphics[width=8cm]{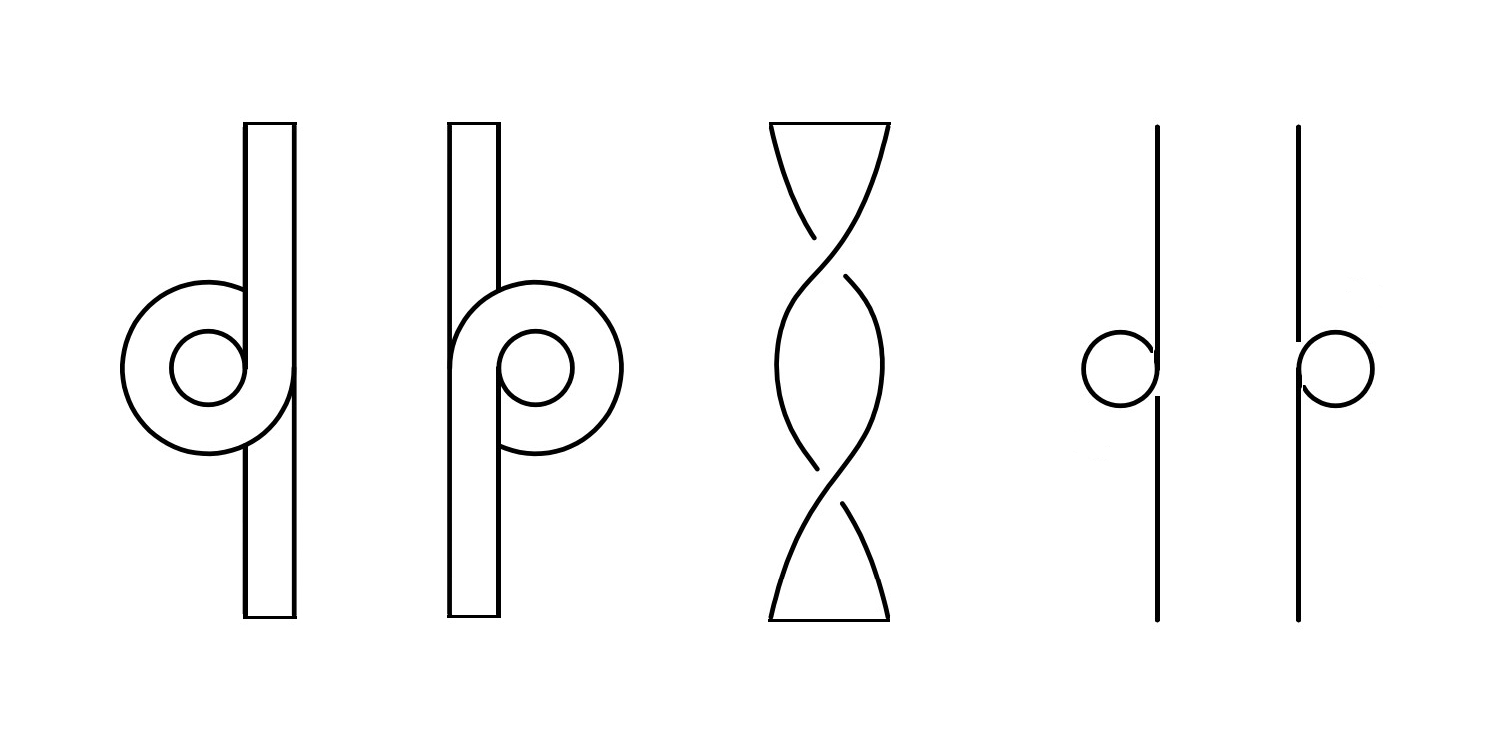}
    \caption{Flattening a positive twist and contracting it to its center-line.}
    \label{fig:flat}
\end{figure} 

A framed link can be equivalently viewed as a link of ribbons instead of solid tori. Indeed, the core curves together with their corresponding longitudes determine the boundary of a link of ribbons in three dimensional space. Conversely, given a link of ribbons one can choose, without any loss of generality, one of the two boundary components as the core curve of a knotted torus and the other boundary component as the corresponding longitude. Any twist of the ribbon, be it positive or negative, corresponds to a full run along the meridian of the solid torus $V$. Since the longitudes follow their respective core curves, all the information of the framed link can be depicted as a link of longitudes with extra curls representing the times a longitude runs along the meridian. See for example Fig. \ref{fig:framed tref}, where $K$ is the core curve, $K_0$ is a longitude that runs once along a meridian in the negative direction. The resulting framed knot is a framed trefoil with framing $w(K)-1=-3-1=-4$.\\
\indent Isotopy of framed link diagrams in $ \mathbb{R}^2 $ consists of the second and third Reidemeister moves, planar isotopies and a modified version of the first Reidemeister move as depicted in Fig. \ref{fig:modRI}. However if we consider diagrams in $S^2$ this move is not necessary as it is the result of sphere isotopies and Reidemeister moves II and III.  Recall that  diagrammatic equivalence  under Reidemeister moves II and III and planar/surface isotopy is named \textit{regular isotopy} due to L.H. Kauffman \cite{kauffman1990invariant}. It is easy to see that framed link isotopy keeps the framing of a knot fixed. The first Reidemeister move is not allowed in framed/ribbon link diagrams because a positive (or negative) twist would result in a $\pm1$ change in the framing. However, two consecutive curls, a positive and a negative, are regular isotopic to an uncurled arc (see Fig. \ref{fig:cancelling}).

\begin{figure}[htp]
    \centering
    \includegraphics[width=6cm]{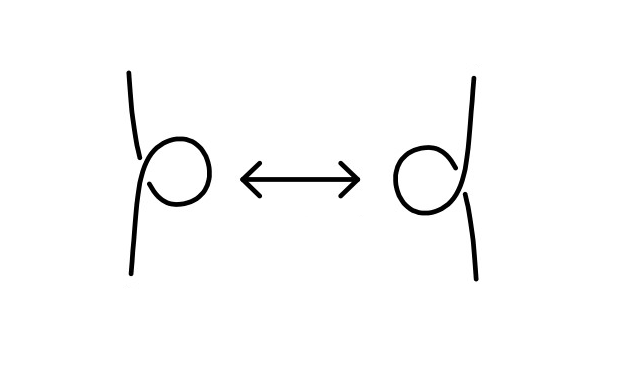}
    \caption{The modified Reidemeister I move.}
    \label{fig:modRI}
\end{figure}

\begin{figure}[htp]
    \centering
    \includegraphics[width=6cm]{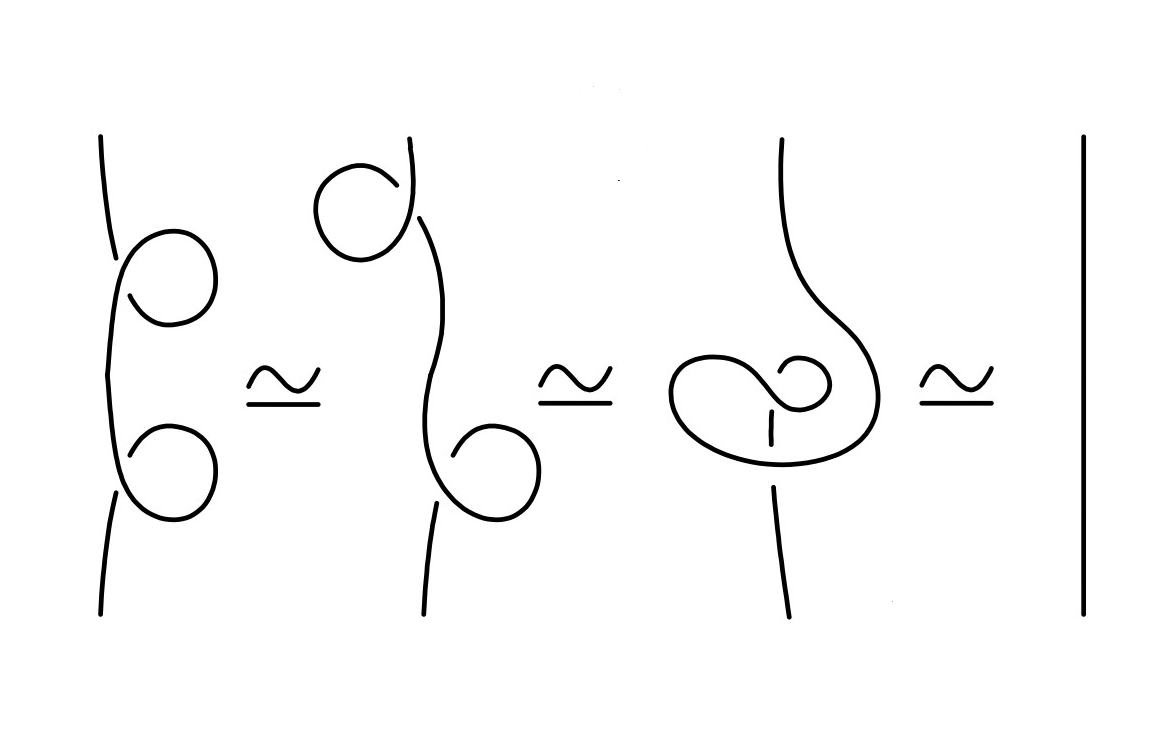}
    \caption{Canceling two opposite curls.}
    \label{fig:cancelling}
\end{figure} 

\subsection{Framed braids}

Geometric framed braids can be described in analogy with classical geometric braids. The role of endpoints in a classical braid will be carried out by two families of intervals in $\Bbb{R}^3$ (or $S^3$), $A_i = [i,i+1/2] \times \{0\} \times \{0\}$ and $B_j = [j,j+1/2] \times \{0\} \times \{1\}$, for $i,j=1,\ldots ,n$. A \textit{band} or a \textit{ribbon} connecting a $B_j$ interval to an $A_i$ for some $i,j \in \{1,\ldots ,n\}$ is a set parametrized as $γ \times [0,1]$, where $γ \subset \Bbb{R}^3$ (or $S^3$) is a braid strand with upper endpoint $\{j+1/4\} \times \{0\} \times \{1\}$ and lower endpoint $\{i+1/4\} \times \{0\} \times \{0\}$. A  \textit{ geometric framed braid} (or \textit{ribbon braid}) on $n$ ribbons is defined as a collection of pairwise non-intersecting $n$ ribbons joining $B_1,\ldots ,B_n$ to $A_1,\ldots ,A_n$ in any order. The main difference from classical braids is that ribbons, unlike strands, can be twisted around their core strand by half twists or full twists. In our case, ribbons are not allowed to have half twists around their core, so if a ribbon twists around its core it must only do so by full twists. Framed braids can be endowed with a natural orientation on the ribbons simply by orienting the core curves downwards (or upwards).

\begin{definition}
The number of full twists a ribbon has is called the \textit{framing of the ribbon}. Furthermore, the {\it framing of a (geometric) framed braid} is defined by a framing vector with the framings of the ribbons as components (see Fig. \ref{fig:FramedBraid}a for an example).    
\end{definition} 

Two geometric framed braids on $n$ ribbons are called \textit{isotopic} if one can be continuously deformed into the other in the class of geometric framed braids, preserving the framing of each ribbon. Isotopy between two geometric braids is an equivalence relation. The equivalence classes are called \textit{framed (or ribbon) braids on $n$ strands}. Again, like in the case of classical braids, we shall use the term \textit{(geometric) framed braid} to mean an equivalence class of framed braids and a concrete representative of such a class. Twists in a framed braid can be isotoped so that they are placed at the top (or bottom) of the framed braid. 

The natural product for classical geometric braids carries over to the geometric framed braids. Indeed, given two geometric framed braids $b_1, b_2$ one can obtain their product by putting them end to end. More precisely, we contract $b_1$ vertically to half its height while keeping the upper endpoint intervals fixed and we also contract $b_2$ vertically to half its height, only this time we keep the lower endpoint intervals fixed. The lower endpoints of $b_1$ will now be of the form $(i,0,1/2)$, same as the upper endpoints of $b_2$. The union of the contracted framed braids is defined as the \textit{product of $b_1$ and $b_2$} and is denoted $b_1b_2$. The equivalence class of this geometric framed braid is defined as the product of the equivalence classes of $b_1$ and $b_2$. This operation turns the set of braids (as equivalence classes) on $n$ strands into a group, with the inverse of a braid $a$ being the mirror image of the (geometric) braid in the plane $\{z=1/2\}$ and the unit element being the framed braid with $n$ vertical untwisted ribbons connecting $B_i$ to $A_i$ for $i=1,\ldots ,n$, so their core curves is the identity in classical braids. Note that the {\it inverse of a ribbon} that twists around its core counterclockwise is the inverse of its core, as a strand of a classical braid, such that as a ribbon it twists counterclockwise around its core. The group of isotopy equivalence classes of framed braids on $n$ strands is called the {\it framed braid group}, denoted $RB_n$.

In order to discretize the theory (as with classical braids), given a framed braid $b$ there is a representative regular diagram that is a projection in $\mathbb{R} \times \{0\} \times [0,1]$ with the following properties: 

\begin{enumerate}
    \item The projections of the ribbons are not tangent to each other and do not overlap with each other unless two ribbons cross transversely.
    \item No crossing region in $\mathbb{R} \times \{0\} \times [0,1]$ is the projection of three or more regions from different ribbons.
    \item Crossing regions that have the same projection in $\mathbb{R} \times \{0\} \times [0,1]$ occur in different heights on the $z$-axis.
    \item A positive twist shall be consistently projected as either one of the first three ribbon diagrams depicted in Fig. \ref{fig:flat}. Analogously,  negative twists shall be consistently projected in the same way but with mirrored crossings and crossing regions.
\end{enumerate}

This projection is called a \textit{ framed braid diagram} of $b$. Two framed braid diagrams representing framed braids on $n$ strands are said to be \textit{equivalent} if they are connected by a finite sequence of plane isotopies that preserve the braid diagram structure and the framed braid isotopy moves  depicted in Fig. \ref{fig:BraidIsotopy}, only this time for flat ribbons instead of strands plus the equivalence depicted in Fig. \ref{fig:flat2}. Projecting twists that way, every framed braid diagram can be visualized lying flat on a plane. Again, like in geometric framed braids, we shall refer to the equivalence classes of framed braid diagrams   and any representative of a class as a \textit{framed braid}. Two geometric framed braids are isotopic if and only if any two of their diagrams are equivalent.

\begin{figure}[htp]
    \centering
    \includegraphics[width=6cm]{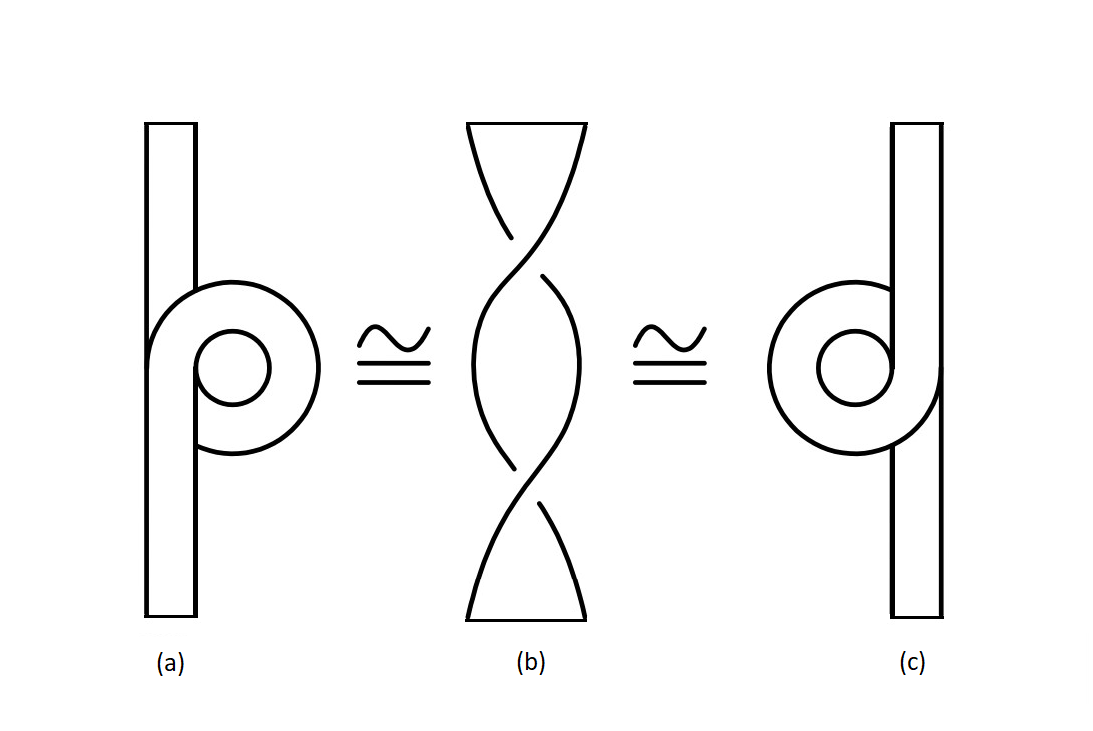}
    \caption{Framed braid diagram isotopy for twists.}
    \label{fig:flat2}
\end{figure} 

Furthermore, there is a way to represent framed braid diagrams as diagrams with classical strands instead of ribbons, via the blackboard framing. By taking a framed braid diagram and projecting every twist flat on the plane of projection, as in Fig. \ref{fig:flat2}, one can contract every ribbon to its center-line, thus portraying ribbons as strands with curls. These curls represent the way a ribbon twists around its center-line and therefore contribute to the blackboard framing, see for example Fig.~\ref{fig:FramedBraid}b. One can also remove the curls from the diagram and attach an integer on each strand representing the total algebraic number of curls on the strand, using framed braid isotopy for collecting the framing to the top of the ribbon. See Fig. \ref{fig:FramedBraid}c for an example.\\
\indent Geometric framed braids form a group and this is also true for framed braid diagrams. Since framed braid diagrams can be represented as classical framed diagrams with curls, the braid groups can be injected into the framed braid groups (any braid can be thought of as a framed braid with zero framing, i.e. with no curls). This means that the classical generators in $B_n$, that we denoted $σ_i$ for $i=1,\ldots ,n-1$, carry over to the framed braid groups, along with their relations. However, additional generators and relations are required to describe the twisting of the ribbons. We denote the additional generators $t_i$, for $i=1,\ldots ,n$ (see Fig \ref{fig:FramedIso}).

\begin{figure}[htp]
    \centering
    \includegraphics[width=8cm]{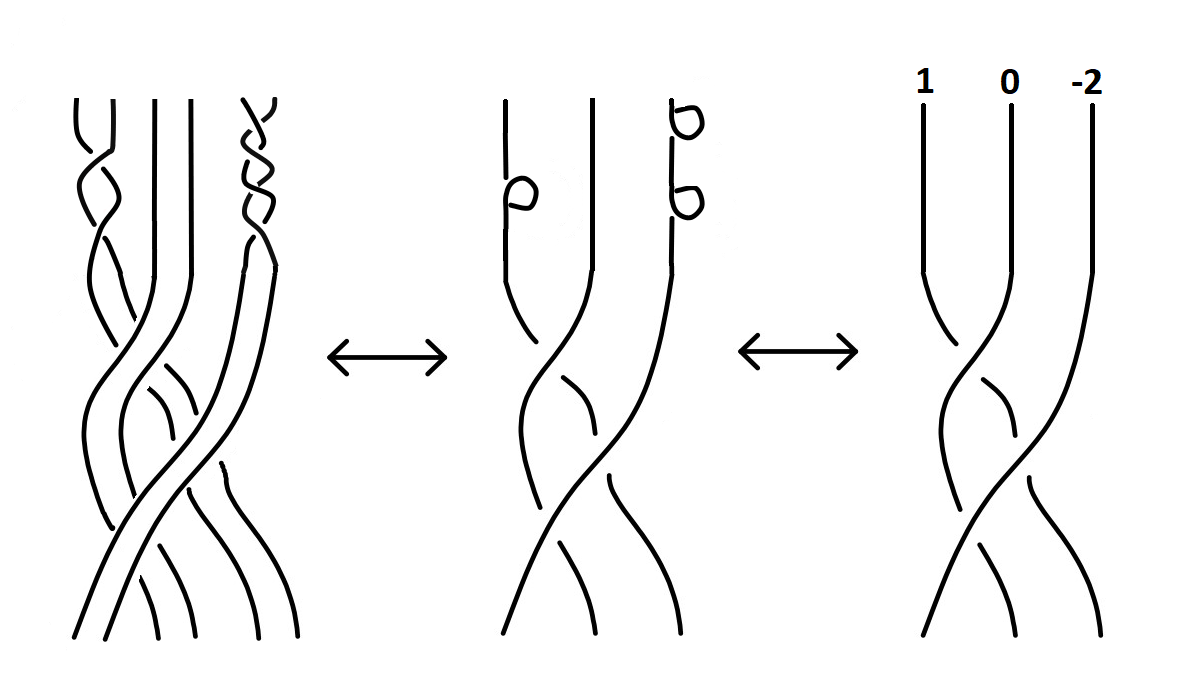}
    \caption{Different equivalent presentations of a framed braid diagram  with three ribbons: (a) as a ribbon diagram; (b) as a classical braid with blackboard framing; and (c) with integer framing (1,0,-2).}
    \label{fig:FramedBraid}
\end{figure} 

\noindent The following is a corollary of Theorem \ref{Artin} (cf.  \cite{ko1992framed} for further details).

\begin{corollary}
The framed braid group of $n$ ribbons will be denoted $RB_n$ and is defined algebraically by the following presentation, where $s_i$ denotes the transposition $(i,i+1) \in S_n$. 

\end{corollary}

\[
RB_n =  
    \Biggl\langle 
       \begin{array}{l|cl}
            σ_1,\ldots ,σ_{n-1} & σ_iσ_jσ_i = σ_jσ_iσ_j &  \forall \, i,j \quad \textrm{so that} \quad |i-j|=1
\\ t_1,\ldots ,t_n     & σ_iσ_j = σ_jσ_i  &  \forall \, i,j \quad \textrm{so that} \quad |i-j| \ge 2 \\
                            & t_it_j=t_jt_i &  \forall  \, i,j \in \{1,\ldots ,n \} \\
                            & σ_it_j = t_{s_i(j)}σ_i &  \forall \, i,j \in \{1,\ldots ,n \}
        \end{array}
     \Biggr\rangle\]

\begin{figure}[htp]
    \centering
    \includegraphics[width=4.5cm]{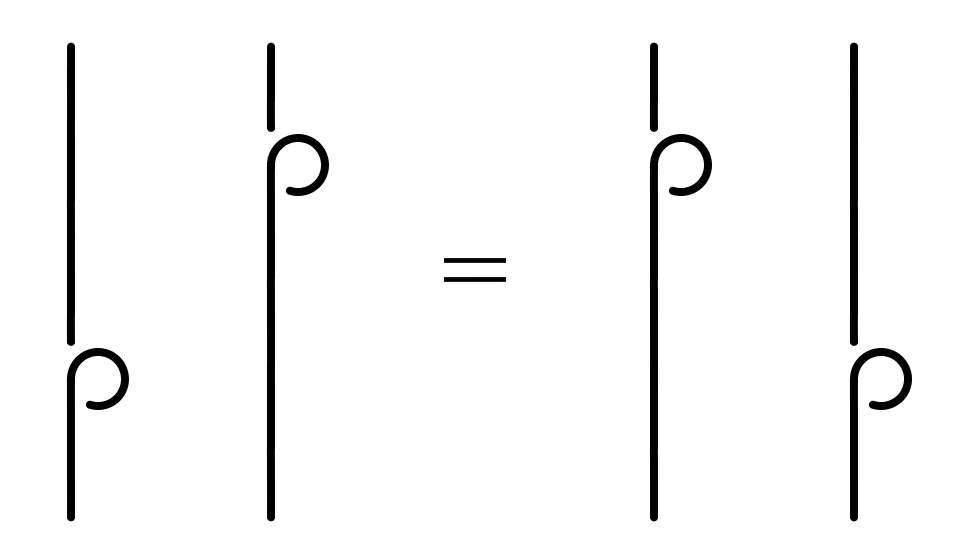}
    \caption{The third family of framed relations in the presentation of $RB_n$.}
    \label{fig:Framedrel2}
\end{figure} 

\begin{figure}[htp]
    \centering
    \includegraphics[width=4cm]{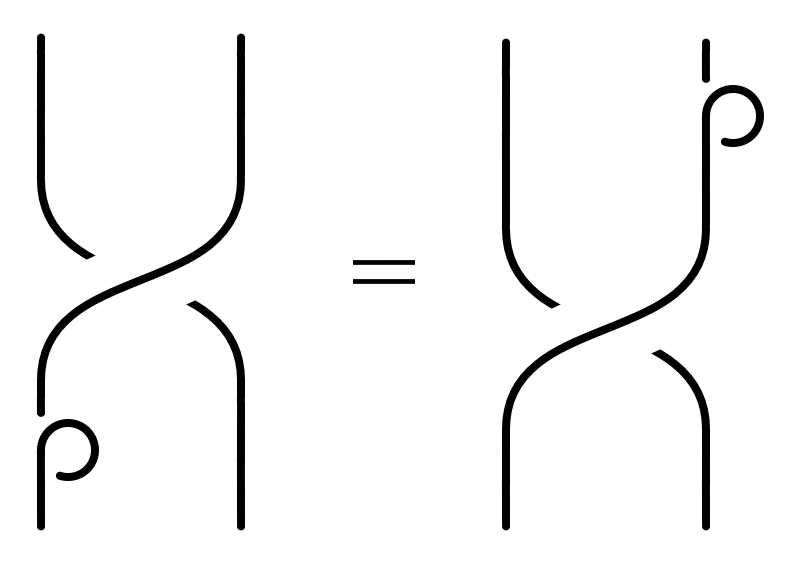}
    \caption{The fourth family of framed relations in the presentation of $RB_n$.}
    \label{fig:Framedrel1}
\end{figure} 

Note that $RB_n$ is isomorphic to the semi-direct product $ \mathbb{Z}^n \rtimes B_n $ also denoted $F_n$. Also note that, because of the fourth relation, any framed braid $b \in RB_n$ can be written in the form:
\\
\\
$$b= t_1^{λ_1}\ldots t_n^{λ_n}β \quad \textrm{where} \quad β \in B_n, \quad λ_1,\ldots ,λ_n \in \mathbb{Z} $$
\\
\\
\begin{figure}[htp]
    \centering
    \includegraphics[width=8cm]{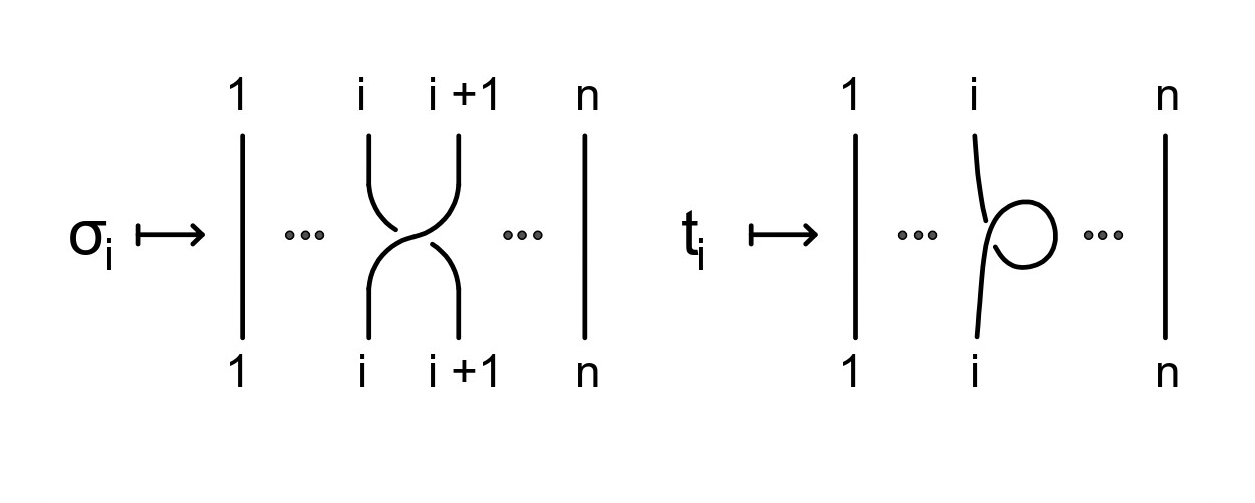}
    \caption{The generators of $RB_n$. The ribbons have been contracted to their center-lines.}
    \label{fig:FramedIso}
\end{figure} 
\\
\\


\section{Relationship between classical links and  braids via the standard closure}
\label{sec:stand_closure}

An important theorem regarding the closure of braids in due to Alexander \cite{alexander1923lemma}. Namely, every type of link can be obtained as the closure of a braid. 

\begin{theorem}[Alexander]
Any oriented link in $\mathbb{R}^3$ (or $S^3$) is isotopic to a closed braid.
\end{theorem}

The main idea behind the proof of this theorem (which is Alexander's original algorithm) is that, given an oriented link diagram and a reference point $o$ on the plane, one can use Reidermeister moves and plane isotopies in order to make the diagram wind counterclockwise (or clockwise) around the chosen reference point. The easiest way to do that is by making use of the fact that every smooth link is isotopic to a piece-wise linear link and vice versa. Then, an edge $e$ of the diagram will be said to wind counterclockwise around $o$ if the oriented triangle that is formed from connecting the endpoints of $e$ with $o$, with respect to the orientation of $e$ in the link, is oriented counterclockwise on the plane. If all edges wind counterclockwise around $o$ then we have a closed braid. If an edge $e'$ winds clockwise around $o$, then the triangle formed in the above way has a clockwise orientation on the plane, derived from the orientation of $e'$ in the link.  We can isotope this triangle with a small planar isotopy so that $o$ is in the interior of the triangle and perform a $Δ$-move (see Fig. \ref{fig:delta}), which induces on link diagrams the same equivalence relation as with the Reidemeister moves and planar isotopy, and replace $e$ with the other two edges of the triangle. Then we have replaced $e'$ with two new edges that wind counterclockwise around $o$. Proceeding this way inductively, we can replace  all edges in the link that wind clockwise around $o$ with edges that wind counterclockwise. By compactness this process will terminate and the result will be a closed braid. Proofs of the Alexander theorem can be found in \cite{alexander1923lemma}.

\begin{figure}[htp]
    \centering
    \includegraphics[width=6cm]{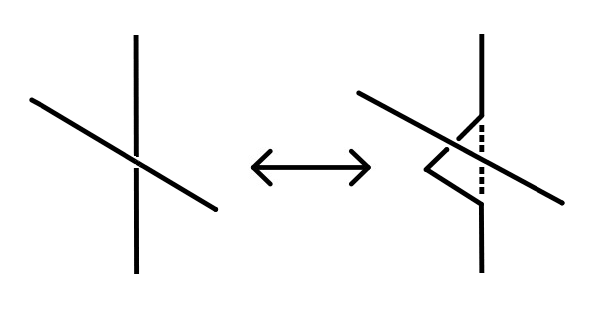}
    \caption{A $Δ$-move.}
    \label{fig:delta}
\end{figure}

Since every (oriented) link is obtained as the closure of a braid, the next natural question is how one can relate two braids, not necessarily with the same number of strands, so that their closures produce isotopic links.

\begin{theorem}[Markov]
\label{Markov}
Given two braids represented by elements $β_1,β_2$ in the braid groups $B_n,B_m$, their closures are isotopic links if and only if $β_2$ can be obtained from applying to $β_1$ a finite sequence of braid relations and the following operations:

\begin{enumerate}
    \item $\quad α \longleftrightarrow βαβ^{-1} \quad \textrm{for} \quad α,β \in B_n \quad \textrm{(conjugation)}$
    \item $\quad α \longleftrightarrow ασ_n^{\pm 1} \quad \textrm{for} \quad α \in B_n \quad \textrm{(M-move)}$
\end{enumerate}

\end{theorem}

\begin{figure}[htp]
    \centering
    \includegraphics[width=5.5cm]{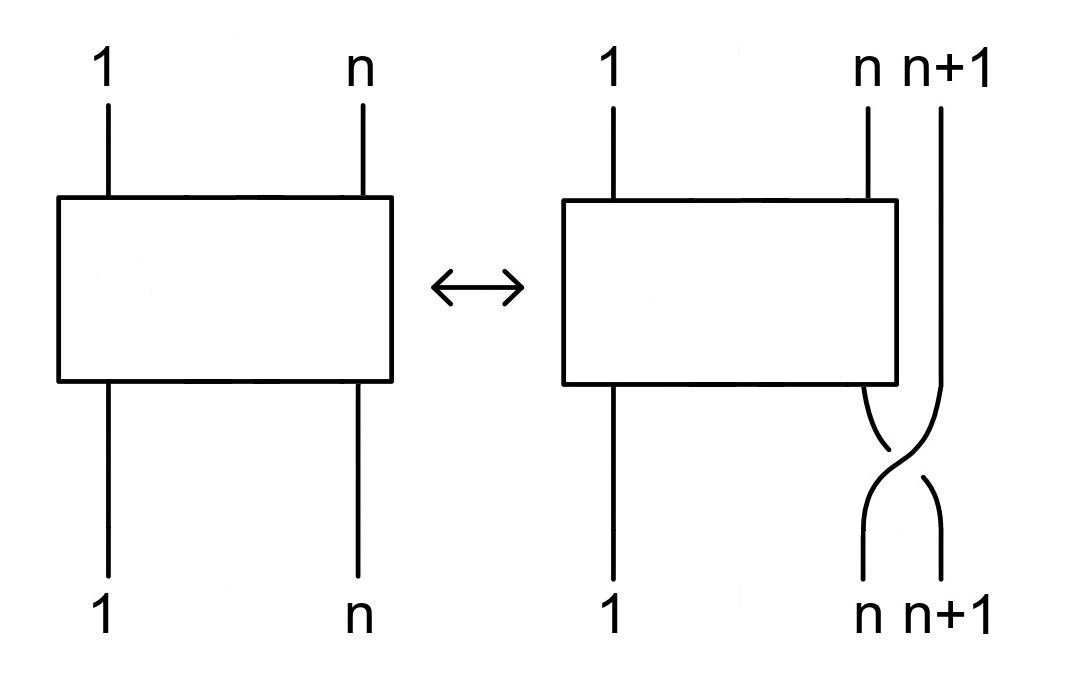}
    \caption{The abstraction of an $M$-move.}
    \label{fig:modRI}
\end{figure}

Proofs of the Markov theorem can be found in \cite{kassel2008braid}. In \cite{lambropoulou1997markov} Lambropoulou and Rourke gave an one-move Markov theorem on braids by introducing a type of move called the $L$-move on braid diagrams.

\begin{definition}[$L$-moves]
Let $D$ be a (piecewise linear) braid diagram and $P$ a point of an arc of $D$ such that $P$ is not vertically aligned with any of the crossings or (other) vertices of $D$ (note that $P$ itself may be a vertex). Then we can perform the following operation: Cut the arc at $P$, bend the two resulting smaller arcs apart slightly by a small isotopy and introduce two new vertical arcs to new top and bottom end-points in the same vertical line as $P$. The new arcs are both oriented downwards and they run either both under or both over all other arcs of the diagram. Thus there are two types of $L$–moves, an under $L$–move or $L_u$–move and an over $L$–move or $L_o$–move. This process is illustrated in Fig.~\ref{fig:L-move}.
\end{definition}

\begin{figure}[htp]
    \centering
    \includegraphics[width=9cm]{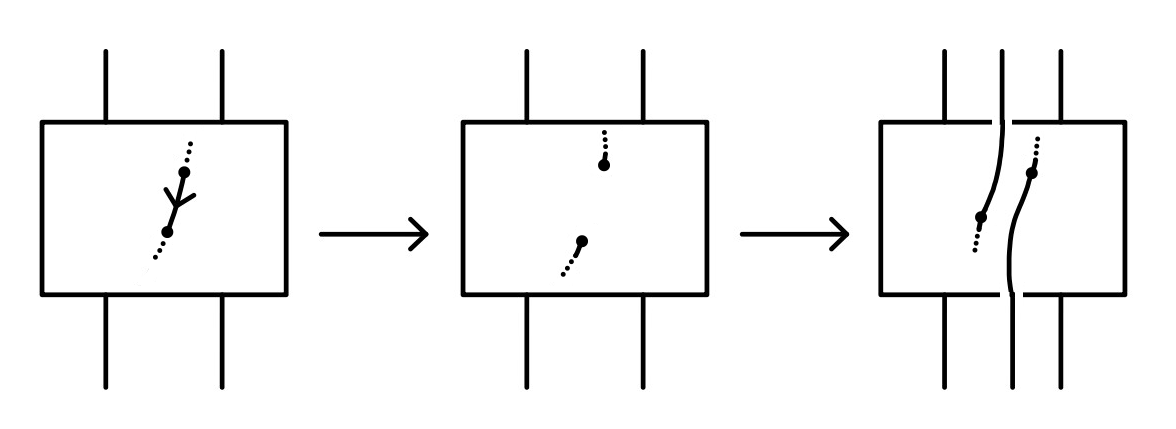}
    \caption{The abstraction of an $L$-over move.}
    \label{fig:L-move}
\end{figure} 

For an equivalent definition of an $L$-move we shall introduce a family of inclusion morphisms, namely $(o_i)_{i=1}^{n+1}$ and $(u_i)_{i=1}^{n+1}$.\\
Here $o_i:B_n \to B_{n+1}$ is a monorphism that acts on $B_n$ by inserting a new strand between the braid inputs/outputs $i-1$ and $i$ of the braid that goes entirely over the braid, see Fig. \ref{fig:morphismoi}\\
\indent We define $u_i$ respectively with the main difference being that the new strand will go entirely under the braid. Notice also that $u_1=o_1$ and that $u_{n+1}=o_{n+1}$ is the natural braid inclusion of $B_n$ into $B_{n+1}$.

\begin{figure}[htp]
    \centering
    \includegraphics[width=8cm]{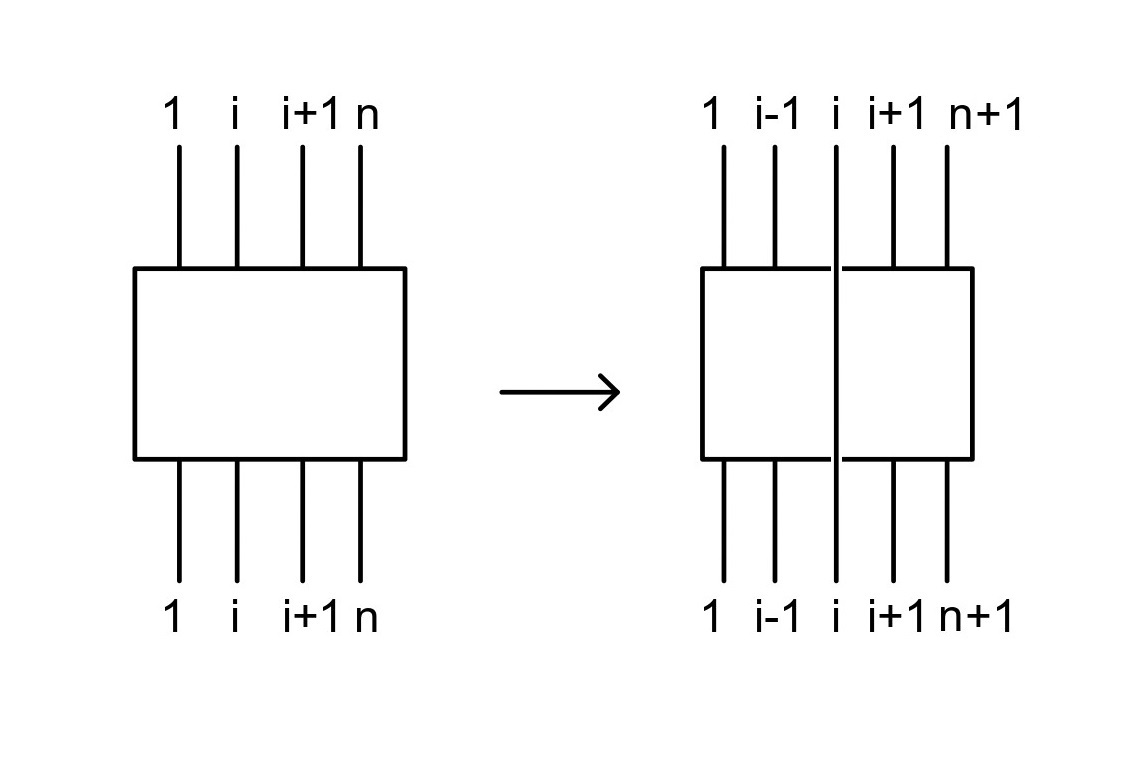}
    \caption{The inclusion morphism $o_i$.}
    \label{fig:morphismoi}
\end{figure}

\begin{definition}[$L$-moves revised]
Let $α \in B_n$ be a braid. We can always assume that the line segment we want to cut is perpendicular to the $x$-axis and that it lies between two braids $α_1,α_2 \in B_n$ so that $α=α_1α_2$. Then an $L_o$-move (standing for $L$-over) will be either one of the following moves given by:

$$o_{i+1}(α_1)σ_i^{\pm 1}o_{i+1}(α_2) \quad \textrm{or} \quad o_{i}(α_1)σ_{i}^{\pm 1}o_i(α_2)$$

\noindent $L_u$ moves are defined similarly by using the family of morphisms $(u_i)$. For a visual presentation of the $L_o$-move see Fig.~\ref{fig:Lo}.\\
\indent We can of course drag the crossing entirely to the right or the left of the picture so that one can get the algebraic expression for the $L_o$-move and $L_u$-move respectively.

\begin{equation}
    α=α_1α_2 \to σ_{i+1}^{-1}\ldots σ_{n}^{-1}o_{n+1}(α_1)σ_{i}^{-1}\ldots σ_{n-1}^{-1}σ_{n}^{\pm 1}σ_{n-1}\ldots σ_{i}o_{n+1}(α_2)σ_n\ldots σ_{i+1}
\end{equation}
\begin{equation}
    α=α_1α_2 \to σ_{i+1}\ldots σ_{n}o_{n+1}(α_1)σ_{i}\ldots σ_{n-1}σ_{n}^{\pm 1}σ_{n-1}^{-1}\ldots σ_{i}^{-1}o_{n+1}(α_2)σ_n^{-1}\ldots σ_{i+1}^{-1}
\end{equation}
\end{definition} 

\begin{figure}[htp]
    \centering
    \includegraphics[width=6cm]{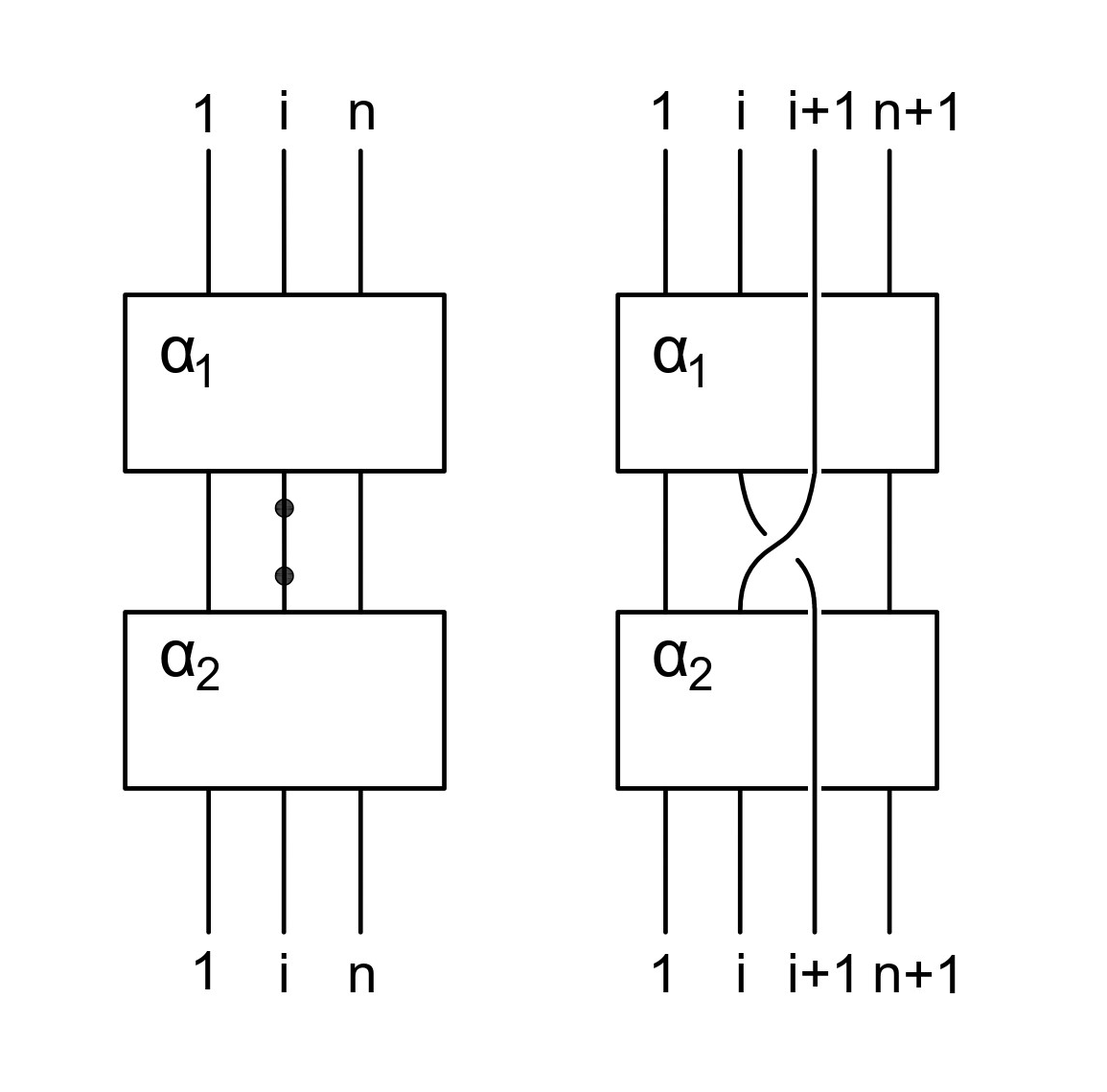}
    \caption{An $L_o$-move introducing a crossing.}
    \label{fig:Lo}
\end{figure}

It is easy to see that the two definitions of an $L$-move are equivalent. Note also that an $M$-move is a special case of an $L$-move. Observe that the closure of two braid diagrams that differ by an $L$-move, represent isotopic links in $S^3$. The $L$-moves induce an equivalence relation on braids so that two braids $b_1$ and $b_2$ are said to be $L$-equivalent if and only if we can transform a braid diagram representative of $b_1$ to a braid diagram representative of $b_2$ by a finite sequence of braid equivalence moves and $L$-moves.  The following theorem is due to Lambropoulou and Rourke \cite{lambropoulou1997markov}:

\begin{theorem}[One-move Markov]
The closure map induces a bijection between the set of $L$–equivalence classes of braids and the set of isotopy types of (oriented) link diagrams.
\end{theorem}

\noindent In \cite{lambropoulou1997markov}, the authors also proved that these types of moves also generate conjugations, so that the Markov theorem is indeed refined.


\section{  Framed links and framed braids via the standard closure}
\label{sec:framed_closure}

Connecting the endpoint intervals of a framed braid diagram on $n$ ribbons with $n$ untwisted and unlinked ribbons results in a framed link. This operation is called the \textit{standard closure of a framed braid}. In particular, the closure of a framed braid is a framed knot if and only if the permutation induced by the framed braid $\in RB_n $  is an  $n$-cycle in the permutation group $S_n$.

\begin{definition}
The \textit{framing of a link component of a closed framed braid} is the sum of the framings of the ribbons that make up this component, plus the writhe of their core curves.  In the case where the closed framed braid framing is a knot, its framing is equal to the sum of all the exponents of the generators $σ_i$ and $t_j$ in the algebraic presentation of the framed braid.    
\end{definition}

\begin{figure}[htp]
    \centering
    \includegraphics[width=6cm]{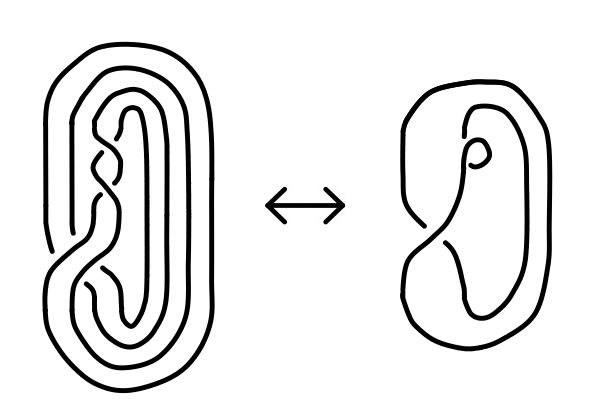}
    \caption{Standard closure of a framed braid.}
    \label{fig:framed ribbons}
\end{figure} 

\subsection{The Alexander theorem for framed links}

The Alexander theorem can be applied to oriented framed links as well, in the same way that it is applied to classical oriented links, so that every oriented framed link is isotopic to the closure of a framed braid. Indeed, observe that any ribbon twist is contained in a small cylindrical neighborhood of the core curve of the ribbon. This means that we can also present curls on the strands as small dots that can travel along the core curves (see Fig. \ref{fig:dots}). This way, we can apply any braiding algorithm to the dotted link as if we had a classical oriented link with the obvious exception that has to be careful not to braid the link on a dot/curl. The resulting framed braid might be different if we move the dots along the link, but this does not affect the result in a harmful way. Indeed, notice that by using conjugations, one can move the framing generators in the braid form of a framed link between strands belonging to the same link component (e.g. $t_1σ_1$ and $σ_1^{-1}t_1σ_1^{2}$ represent the same framed link when closed, however notice that $σ_1^{-1}t_1σ_1^{2}=t_2σ_1^{-1}σ_1^{2}=t_2σ_1$ due to the fourth family of relations in the group presentation). From the discussion above, we have:
\begin{theorem}[Framed Alexander]
Any oriented framed link in $\mathbb{R}^3$ is isotopic to a closed framed braid.    
\end{theorem}

\begin{figure}[htp]
    \centering
    \includegraphics[width=3cm]{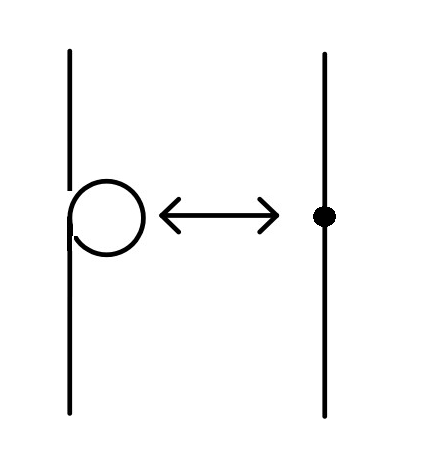}
    \caption{Replacing curls with dots.}
    \label{fig:dots}
\end{figure} 

\subsection{The Markov theorem for framed braids}

There is a direct analogue to the Markov theorem for framed braids.

In order to modify the classical Markov theorem for framed braids we need to account for any change of framing at any point in the sequence when we modify the braid. Conjugation does not affect the framing of a braid, but an $M$-move changes it by $\pm 1$. So the $M$-move (recall Theorem \ref{Markov}) will have to be adapted to the followinf the \textit{RM-move}:

$$(RM) \quad α \longleftrightarrow αt_n^{\mp 1}σ_n^{\pm 1} \longleftrightarrow ασ_n^{\pm 1}t_{n+1}^{\mp 1} \quad \textrm{for} \quad α \in RB_n$$

\begin{theorem}[Framed Markov]
\label{Framed Markov}
Given two framed braids $b_1,b_2$ in the braid groups $RB_n,RB_m$ represented by the braid diagrams $d_1,d_2$ respectively, their closures are isotopic framed links if and only if $d_2$ can be obtained from applying to $d_1$ a finite sequence of braid relations, conjugations and $RM$-moves. 
\end{theorem}

A proof of a variant of Theorem~\ref{Framed Markov} is given in \cite{ko1992framed}. Smolinsky and Ko's proof uses integer framing instead of blackboard framing, whereas the framing of a braid strand is indicated by an integer (see Fig.~\ref{fig:FramedBraid}).  Recall that a framed oriented knot $K$ with framing $m$ is essentially a knotted solid torus $V \cong S^1 \times D^2$ embedded in $S^3$, where the core curve $S^1 \times \{0 \}$ is our oriented knot $K$, along with a simple closed curve $K_0 \subset \partial V$ where $K_0$ is homologous to the element $m \in H_1(S^3 \setminus K)$. Recall also that framed knot ambient isotopy is the ambient isotopy of $(V,K_0)$ inside $S^3$. The difference between blackboard and integer framing is essentially a matter of representation, as a blackboard framed diagram depicts the parallel curve in $\partial V$ whereas an integer framed diagram depicts the core curve. 
 In that version of Theorem~\ref{Framed Markov} the $RM$-move is replaced by the classical $M$-move, with zero framing on the new strand, since in that setting the extra crossing does not change the framing of the strand involved.

\subsection{Ribbon or framed $L$-moves}

In this section we shall introduce the framed $L$-moves. Similarly to the $M$-moves, the classical $L$-moves change the framing when performed on a ribbon braid diagram. So, a similar modification as in the $Μ$-move case must be implemented. Namely we have:

\begin{definition}[Ribbon $L$-moves or $RL$-moves]
An {\it $RL$-move} is the framed analogue of an $L$-move. More precisely let $a,a_1,a_2 \in RB_n$ be ribbon braid diagrams so that $a=a_1a_2$. Then an $RL_{o}$-move is one of the following moves:

$$a=a_1a_2 \quad \longleftrightarrow \quad o_{i+1}(a_1)t_i^{\mp 1}σ_{i}^{\pm 1}o_{i+1}(a_2) \quad \textrm{or} \quad o_{i}(a_1)t_{i+1}^{\mp 1}σ_{i}^{\pm 1}o_i(a_2) $$
where $o_i$ is the morphism that introduces a new strand that passes over the (framed) braid between the $i$ and $i+1$ ends. View Fig.~\ref{fig:RLo}. Analogously, an $RL_{u}$ move is described as one of the following moves:
$$a=a_1a_2 \quad \longleftrightarrow \quad u_{i+1}(a_1)t_i^{\mp 1}σ_{i}^{\pm 1}u_{i+1}(a_2) \quad \textrm{or} \quad u_{i}(a_1)t_{i+1}^{\mp 1}σ_{i}^{\pm 1}u_i(a_2) $$

\begin{figure}[htp]
    \centering
    \includegraphics[width=5.5cm]{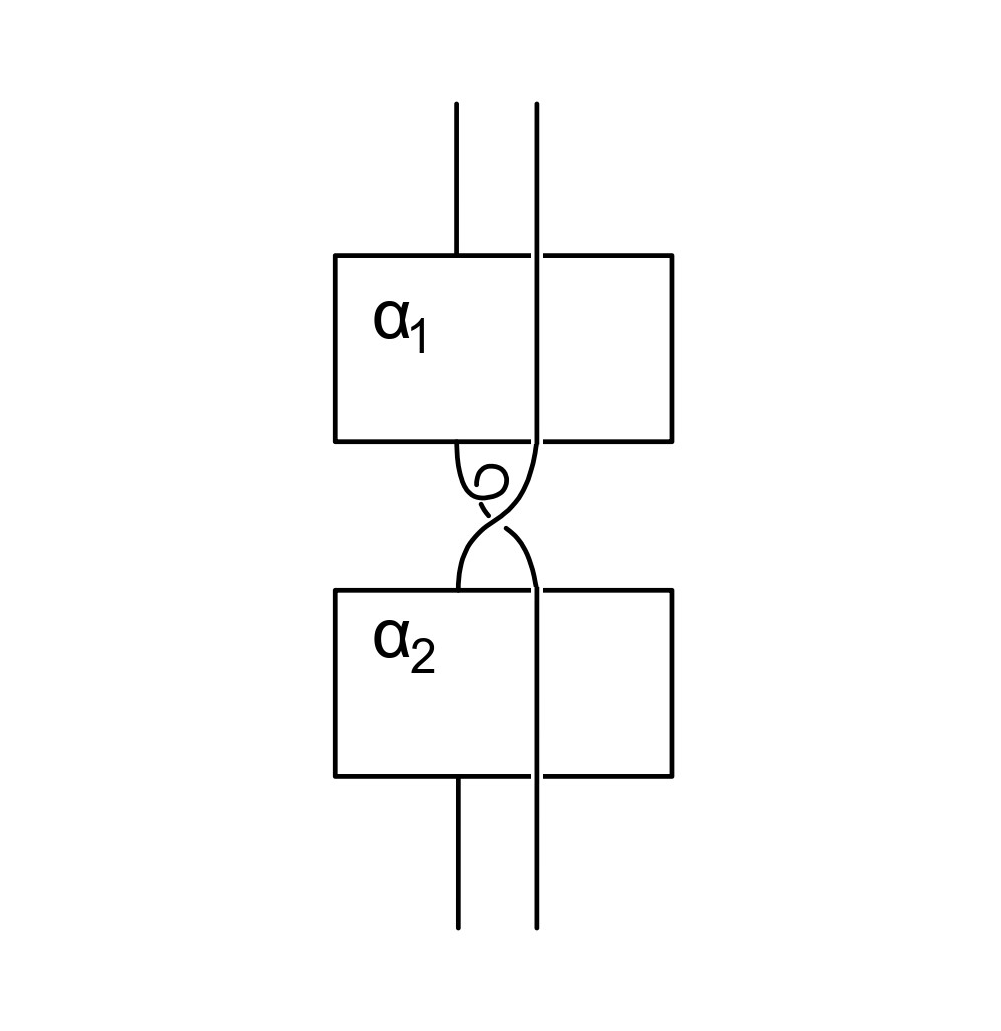}
    \caption{An $RL_o$-move}
    \label{fig:RLo}
\end{figure} 

The reason for adding a framing generator to an $RL$-move is because a standard $L$-move introduces a new crossing and therefore changes the writhe, and subsequently the blackboard framing, of the framed component when applying the standard closure. With an $RL$-move the change in framing that the new crossing generator introduces is neutralized by the addition of the framed generator with opposite sign.

Moving the crossing and curl parts that an $RL$-move introduces to the right (or to the left) of the diagram, via a framed braid isotopy we obtain the following {\it algebraic expressions for the $RL$-moves}:

\begin{equation}
    α=α_1α_2 \to σ_{i+1}^{-1}\ldots σ_{n}^{-1}o_{n+1}(a_1)σ_{i}^{-1}\ldots σ_{n-1}^{-1}t_i^{\mp 1}σ_{n}^{\pm 1}σ_{n-1}\ldots σ_{i}o_{n+1}(a_2)σ_n\ldots σ_{i+1}
\end{equation}
\begin{equation}
    α=α_1α_2 \to σ_{i+1}\ldots σ_{n}o_{n+1}(a_1)σ_{i}\ldots σ_{n-1}t_i^{\mp 1}σ_{n}^{\pm 1}σ_{n-1}^{-1}\ldots σ_{i}^{-1} o_{n+1}(a_2)σ_n^{-1}\ldots σ_{i+1}^{-1}
\end{equation}
\end{definition}

\subsection{Framed $L$-equivalence}

We now state the following result, which we will prove without assuming Theorem~\ref{Framed Markov}:

\begin{theorem}[Framed $L$-equivalence]
\label{framed L-Markov}
The closure of two framed braids $b_1,b_2$ (not necessarily with the same number of strands) induces isotopic framed links if and only if $b_1$ and $b_2$ are related by framed braid isotopies and  a series of $RL$-moves.
\end{theorem}

\noindent To prove Theorem~\ref{framed L-Markov} we first prove the following lemma:

\begin{lemma}
\label{Framed conj}
    A conjugation by a word generated by framed generators only can be realized as a finite sequence of $RL$-moves and planar isotopies.
\end{lemma}

\begin{proof}
The diagrammatic argument in Fig. \ref{fig:framed conjugation} gives us the sequence of $RL$-moves: $$t_i^{-1}at_i \to t_i^{-1}o_{i+1}(1_n)t_iσ_i^{-1}o_{i+1}(at_i) \to o_i(a)t_{i+1}σ_i^{-1}ο_i(1_n)\to a $$
\end{proof}

\begin{figure}[htp]
    \centering
    \includegraphics[width=13cm]{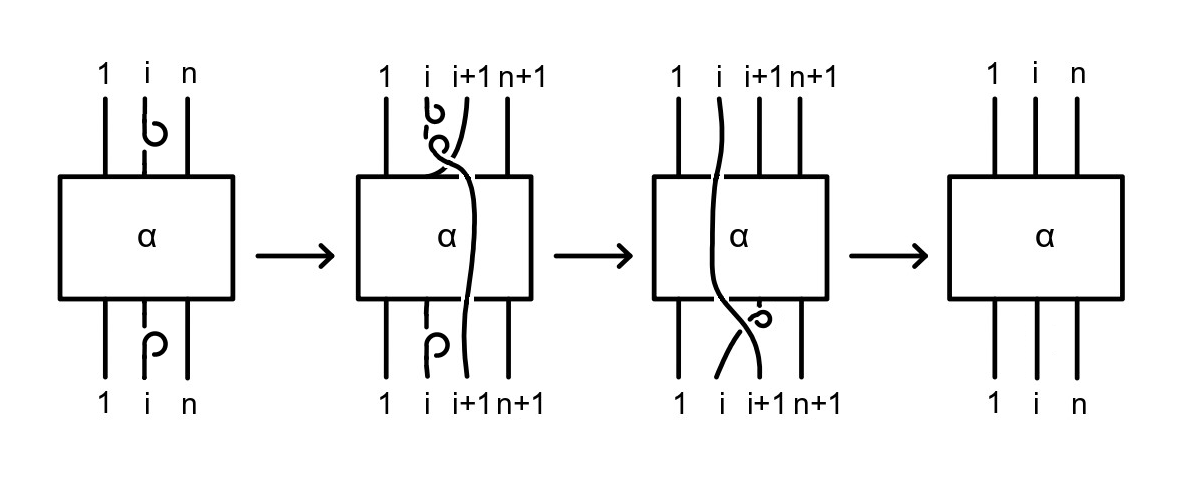}
    \caption{A conjugation involving a framing generator as a sequence of $RL$-moves}
    \label{fig:framed conjugation}
\end{figure}

\begin{proof}[Proof of Theorem \ref{framed L-Markov}]
For the `if' part, it is easy to see that the closure of a framed braid with an $RL$-move performed, for example: $$σ_{i+1}^{-1}\ldots σ_{n}^{-1} \hat{α_1}σ_{i}^{-1}\ldots σ_{n-1}^{-1}t_i^{\mp 1}σ_{n}^{\pm 1}σ_{n-1}\ldots σ_{i} \hat{α_2}σ_n\ldots σ_{i+1}$$ will result in two opposite curls that cancel each other out and the resulting framed link is regular isotopic to the closure of $α_1α_2$.

For the `only if' part let us take two framed braids $b_1,b_2$ which give isotopic framed link diagrams. Due to the relations in the framed braid group we could write 

$$b_1= t_1^{λ_1}\ldots t_n^{λ_n}β_1 \quad \textrm{where} \quad β_1 \in B_n$$
$$b_2= t_1^{κ_1}\ldots t_m^{κ_m}β_2 \quad \textrm{where} \quad β_2 \in B_m$$

\noindent By setting $$π_k:RB_k \to B_k$$ $$π_k(t_i) = 1, \quad \forall  i=1,\ldots ,k$$ $$π_k(σ_j)=σ_j \quad \forall j=1,\ldots ,k-1$$ 
we have a map $$π:\cup_nRB_n \to \cup_nB_n$$Since the closures of $b_1,b_2$ give isotopic framed links, this means that $π(b_1)$ and $π(b_2)$ are connected by 
a series of $L$-moves and braid isotopy moves, for they generate isotopic link diagrams when closed. We will use this sequence of $L$-moves to formulate 
a sequence of $RL$-moves for our framed braids. The construction goes as such. Whenever a braid isotopy move is performed, we perform the exact same move in the ribbon braids. Whenever an $L$-move is performed in the sequence, we perform an $RL$-move 
in the ribbon braids. Thus we have created a sequence of $RL$-moves connecting the framed braid diagrams $b_1=t_1^{λ_1}\ldots t_n^{λ_n}β_1$ and $b'_2=t_1^{δ_1}\ldots t_m^{δ_m}β_2$. Because the closure of $b_2$ and $b'_2$ represents isotopic framed links we have that $\sum_{i=0}^{2m}δ_i = \sum_{i=0}^{2m}κ_i$.\\
\indent We now need to modify the framing generators in $b'_2$ accordingly in order to get the framing generators of $b_2$. We can do that by using conjugations. It is a well known fact that two permutations are conjugate if and only if they have the same cycle structure in their analysis in disjoint cycles. So by using appropriate conjugations with framed braids of $0$ framing on the strands we can assume, without any loss in generality, that the projection of $β_2$ (or of $b_2,b_2'$ since the projection to the symmetric group forgets the framing) is the permutation $(1,\ldots ,s_1)(s_1+1,\ldots ,s_1+s_2)\ldots (s_1+\ldots +s_{p-1}+1,\ldots ,s_1+\ldots +s_p)$, for some positive integers $s_1,\ldots ,s_p$ where $s_1+\ldots +s_p=m$. Since these conjugations involve only the classical braid generators we can relate the projections of the framed braids, before and after applying the conjugations, by a finite sequence of $L$-moves. We then apply on the framed braids $b_2,b_2'$ the sequence of $RL$-moves corresponding to the sequence of $L$-moves describing the conjugations with $0$-framing.

\noindent Let $γ=t_1^{r_1}\ldots t_m^{r_m}$. $$γ^{-1}b_2=b'_2γ^{-1} \Leftrightarrow t_1^{λ_1-r_1}\ldots t_m^{λ_m-r_m}β_2 = t_1^{λ_1-r_{β_2(1)}}\ldots t_m^{λ_m-r_{β_2(m)}}β_2$$ $$ \Leftrightarrow 
t_1^{λ_1-r_1}\ldots t_m^{λ_m-r_m} = t_1^{δ_1-r_{β_2(1)}}\ldots t_m^{δ_m-r_{β_2(m)}} \Leftrightarrow $$ $$δ_i-r_i = δ_i-r_{β_2} \quad \forall i=1,\ldots ,m$$ since the $t_i$'s produce a free abelian group in $RB_m$. It's easy now to find closed integer solutions for this system of equations breaking it into smaller systems according to the permutation cycles.

\noindent Using Lemma \ref{Framed conj} the proof is complete.
\end{proof}

\begin{remark}
    If we are to rephrase Theorem~\ref{framed L-Markov} in the spirit of integer framing of  Smolinsky and Ko \cite{ko1992framed} instead of blackboard framing, then all we need is to consider $L$-moves for integer framed braids or for braids of solid cylinders. 
\end{remark}

\noindent We give the following definition:

\begin{definition}[Integer $RL$-moves]
Let $D$ be an integer framed braid diagram and $P$ a point of an arc of $D$ such that $P$ is not vertically aligned with any of the crossings of $D$. Then we can perform the following operation: Cut the arc at $P$, bend the two resulting smaller arcs apart slightly by a small isotopy and introduce two new vertical arcs to new top and bottom end-points in the same vertical line as $P$. Then, appoint two additional framings $k$ and $-k$ on the new top and bottom endpoints respectively ($k \in \{-1,0,1\}$). The new arcs are both oriented downwards and they run either both under or both over all other arcs of the diagram. Algebraically these moves correspond to the words: \begin{equation}
    α=α_1α_2 \to t_{i+1}^kσ_{i+1}^{-1}\ldots σ_{n}^{-1}o_{n+1}(α_1)σ_{i}^{-1}\ldots σ_{n-1}^{-1}σ_{n}^{\pm 1}σ_{n-1}\ldots σ_{i}o_{n+1}(α_2)σ_n\ldots σ_{i+1}t_{i+1}^{-k}
\end{equation}
\begin{equation}
    α=α_1α_2 \to t_{i+1}σ_{i+1}\ldots σ_{n}o_{n+1}(α_1)σ_{i}\ldots σ_{n-1}σ_{n}^{\pm 1}σ_{n-1}^{-1}\ldots σ_{i}^{-1}o_{n+1}(α_2)σ_n^{-1}\ldots σ_{i+1}^{-1}t_{i+1}^{-k}
\end{equation}
\end{definition}

\noindent This operation can be interpreted as cutting an open cylindrical neighborhood on a strand of a braid of solid cylinders and dragging the new boundary disks over or under any other cylindrical strand onto two newly created corresponding disk ends, while twisting the upper end $k$ times and twisting the lower end $-k$ times, for $k \in \{-1,0,1\}$. It is easy to see (Fig.~\ref{fig:integer framed conjugation}) that these moves can generate conjugations involving only framing generators and thus we have the following version of Theorem~\ref{framed L-Markov}.

\begin{figure}[htp]
    \centering
    \includegraphics[width=13cm]{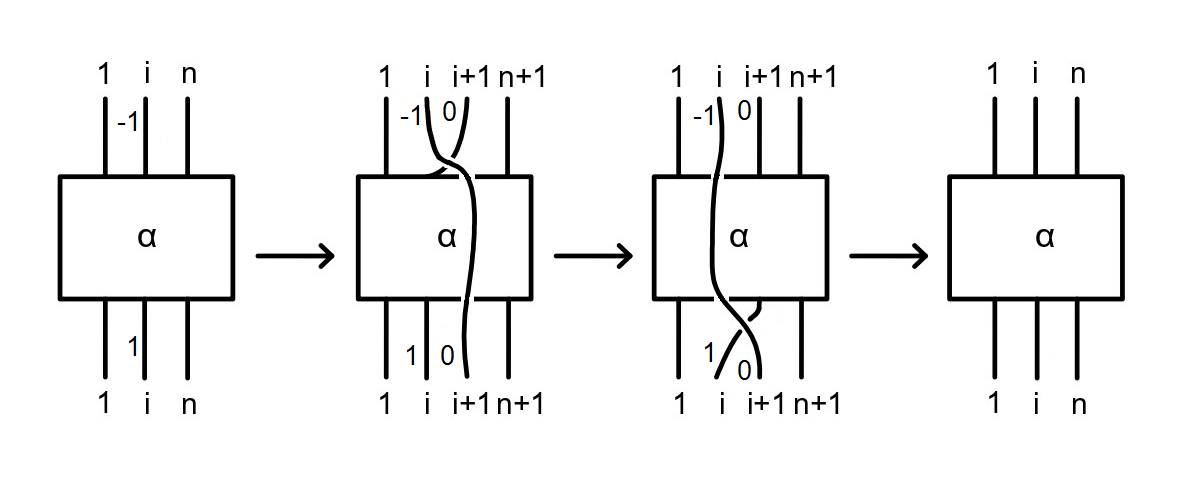}
    \caption{A conjugation involving a framing generator as a sequence of integer $RL$-moves}
    \label{fig:integer framed conjugation}
\end{figure}

\begin{theorem}
        The closure of two integer framed braids $b_1,b_2$ (not necessarily with the same number of strands) induces isotopic integer framed links if and only if $b_1$ and $b_2$ are related by framed braid isotopies and a series of integer $RL$-moves.
\end{theorem}


\section{The plat closure, the Hilden groups $H_{2n}$ and the plat equivalence}
\label{sec:Hilden_plat}

Let $M$ be a connected oriented topological manifold (possibly with boundary) and let $Q$ be a finite (possibly empty) subset in the interior of $M$, denoted $intM$. By a \textit{self-homeomorphism} of the pair $(M,Q)$ we mean a homeomorphism $f:M \to M$, so that $f$ fixes $\partial M$ point-wise and fixes $Q$ set-wise. If $\partial M = \emptyset$ then we also need $f$ to preserve the orientation of $M$. Two self-homeomorphisms $f,g$ of $(M,Q)$ are said to be isotopic if there is a family $\{f_t\}$, with $t \in [0,1]$, of self-homeomorphisms of $(M,Q)$ so that $f_0=f$ and $f_1=g$ and the map $F:M \times [0,1] \to M$, where $F(x,t)=f_t(x)$ is continuous. Isotopy of self-homeomorphisms induces an equivalence relation. Let $G$ be the group of all self-homeomorphisms of $(M,Q)$, with map composition for group multiplication and let $H$ be the normal subgroup of $G$ of all self-homeomorphisms of $(M,Q)$ that are isotopic to the identity. The group $G/H$ is denoted MCG$(M,Q)$ and called the {\it mapping class group} of $(M,Q)$. 

 It is a well known result that the braid group $B_n$ is isomorphic to MCG$(D^2,Q)$, where $D^2$ is the standard two-dimensional disc in $\Bbb{R}^2$ and $Q$ is a set of $n$ distinct points in the interior of $D^2$. Let $\Bbb{R}^3_+ = \{(x,y,z) \in \Bbb{R}^3|z \ge 0 \}$ and $\{a_i\}_{i=1}^n$ be a family of properly embedded, pairwise disjoint unknotted arcs in $\Bbb{R}^3_+$ so that $\partial a_i \in \partial \Bbb{R}^3_+$ for all $i=1,\ldots ,n$ and denote $a^*_n = a_1 \cup \ldots  \cup a_n$. Then MCG$(D^2, \partial a^*_n)=B_{2n}$, where $D^2$ is canonically embedded in $\partial \Bbb{R}^3_+$ and $\partial a^*_n$ lies in the interior of $D^2$. The subgroup of $B_{2n}$ consisting of the equivalence classes of self-homeomorphisms of $(D^2, \partial a^*_n)$ that can be extended to self-homeomorphisms of $(\Bbb{R}^3_+,a^*_n)$ is called the \textit{Hilden Group on 2n strands} and is denoted $H_{2n}$ (see \cite{hilden1975generators}). Equivalently, $H_{2n}$ is the subgroup of $B_{2n}$ of all the braids that stabilize  $a^*_n$ realized as a $(0,2n)$-tangle (see Fig. \ref{fig:Hilden1}), through the action of $Β_{2n}$ on $(0,2n)$-tangles that places the diagram of $a^*_n$ on top of a braid diagram. In other words the resulting $(0,2n)$-tangle is isotopic to the tangle $a^*_n$. For an example of a braid diagram in $H_4$ see Fig.~\ref{fig:Hilden_braid}.

 Hilden in \cite{hilden1975generators} gave generators for the Hilden group $H_{2n}$. Tawn in \cite{tawn2007presentation} showed that the Hilden group on $2n$ strands is generated by the elements (see Fig.~\ref{fig:Hilden_generators}): \begin{align*}
    P_i &= σ_{2i}σ_{2i-1}σ_{2i+1}^{-1}σ_{2i}^{-1} \quad \textrm{for} \quad 1 \le i \le n-1\\
    S_j &= σ_{2j}σ_{2j-1}σ_{2j+1}σ_{2j} \quad \textrm{for} \quad 1 \le j \le n-1\\
    Θ_k &= σ_{2k-1} \quad \quad \quad \quad \quad \quad \textrm{for} \quad 1 \le k \le n
\end{align*} subject to relations:

\begin{equation}
 \begin{aligned}
 \label{Hilden relations1}
     P_iP_j &= P_jP_i \quad \quad \textrm{for} \quad |i-j| > 1\\ 
    P_iP_jP_i &= P_jP_iP_j \quad \textrm{for} \quad |i-j| = 1\\ 
    S_iS_j &= S_jS_i \quad \quad \textrm{for} \quad |i-j| > 1\\ 
    S_iS_jS_i &= S_jS_iS_j \quad \textrm{for} \quad |i-j| = 1\\ 
    P_iS_j &= S_jP_i \quad \quad \textrm{for} \quad |i-j| > 1\\
     P_iP_j &= P_jP_i \quad \quad \textrm{for} \quad |i-j| > 1\\ 
     P_iP_jP_i &= P_jP_iP_j \quad \textrm{for} \quad |i-j| = 1\\ 
     S_iS_j &= S_jS_i \quad \quad \textrm{for} \quad |i-j| > 1\\ 
     S_iS_jS_i &= S_jS_iS_j \quad \textrm{for} \quad |i-j| = 1\\ 
     P_iS_j &= S_jP_i \quad \quad \textrm{for} \quad |i-j| > 1 \\
     P_iS_{i+1}S_i &= S_{i+1}S_iP_{i+1} \\
     P_{i+1}P_iS_{i+1} &= S_iP_{i+1}P_i \\
     P_{i+1}S_iS_{i+1} &= S_iS_{i+1}P_i \\
     P_iΘ_iS_iP_i &= S_iΘ_i \\
     P_iΘ_j &= Θ_jP_i \quad \quad \textrm{for} \quad j \neq i \quad \textrm{or} \quad i+1\\
     P_iΘ_{i+1} &= Θ_iP_{i+1} \\
     S_iΘ_j &= Θ_jS_i \quad \quad \textrm{if} \quad j \neq i \quad \textrm{or} \quad i+1\\
     S_iΘ_j &= Θ_kS_i \quad \quad \textrm{if} \quad \{i,i+1\} = \{j,k\}\\
     Θ_iΘ_j &= Θ_jΘ_i \quad \quad \textrm{for} \quad 1 \le i,j \le n
 \end{aligned}
\end{equation}

\begin{figure}[htp]
    \centering
    \includegraphics[width=13cm]{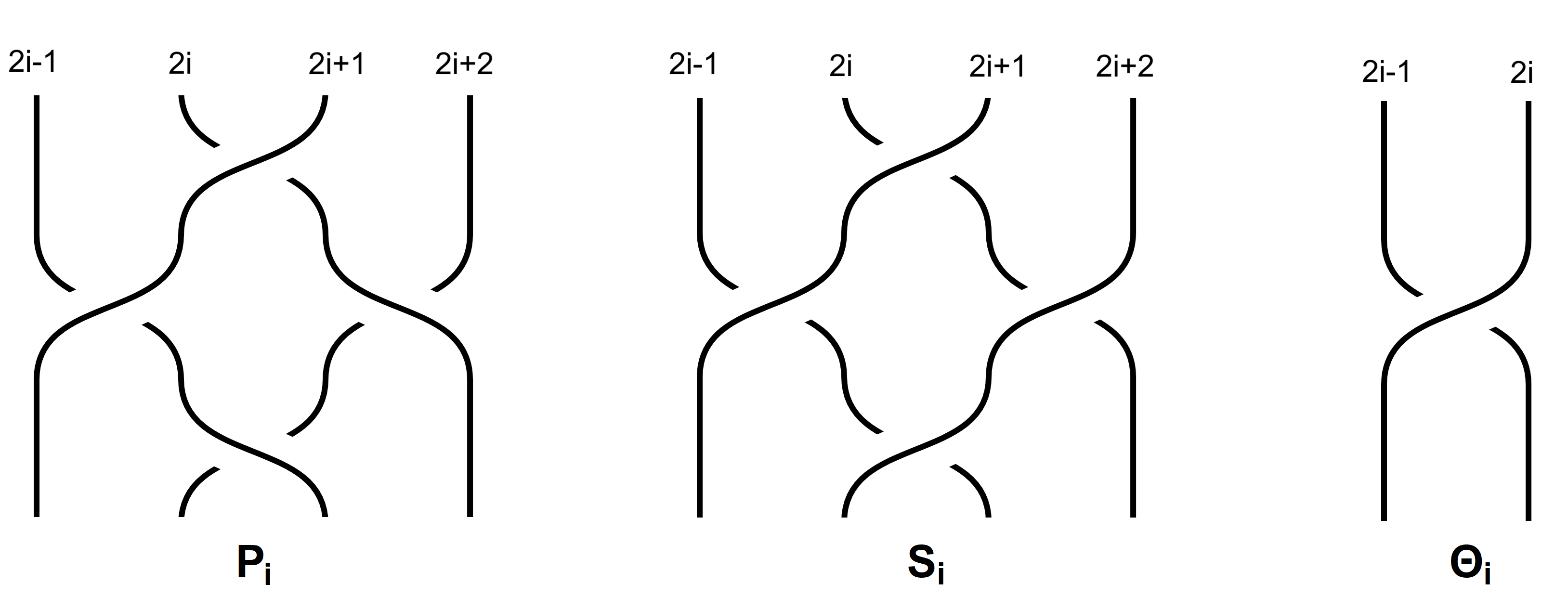}
    \caption{The Hilden group generators: $P_i, S_i$ and $Θ_i$.}
    \label{fig:Hilden_generators}
\end{figure}

\begin{figure}[htp]
    \centering
    \includegraphics[width=4.5cm]{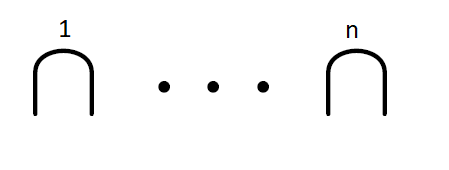}
    \caption{Diagrammatic representation of  $a^*_n$.}
    \label{fig:Hilden1}
\end{figure}

\begin{figure}[htp]
    \centering
    \includegraphics[width=3cm]{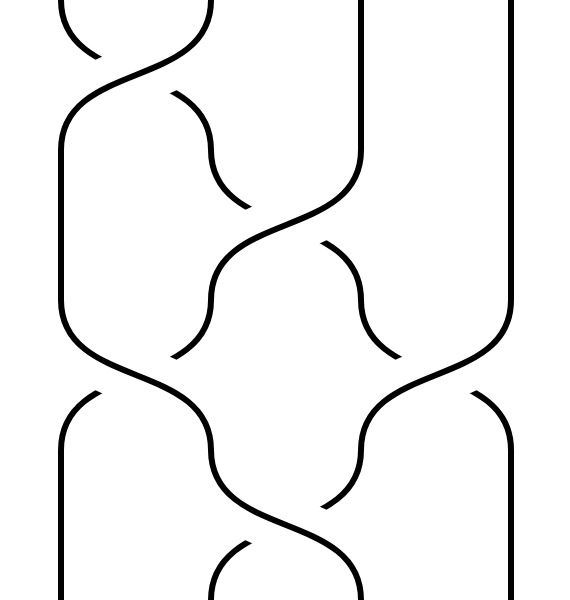}
    \caption{The braid $σ_1σ_2σ_3σ_1^{-1}σ_2^{-1}$ is an element of $H_4$.}
    \label{fig:Hilden_braid}
\end{figure}

\begin{definition}
    Let  $a^{*-}_n$ be the diagram depicted in Fig. \ref{fig:Hilden2} and let $β \in B_{2n}$. Then the diagram $a^*_n \, β \, a^{*-}_n$ is an unoriented link diagram. The ambient isotopy class of the link corresponding to the equivalence class of $a^*_n \, β \, a^{*-}_n$ is called the \textit{plat closure of $β$}.
\end{definition}  

\begin{figure}[htp]
    \centering
    \includegraphics[width=4.5cm]{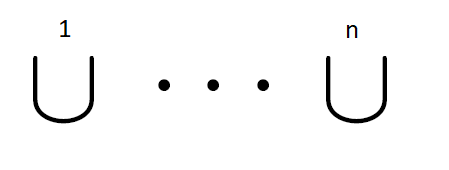}
    \caption{Diagrammatic representation of $a^{*-}_n$.}
    \label{fig:Hilden2}
\end{figure}

    A theorem of Hilden and Birman tells us that any link $L$ can be realized as the plat closure of a braid in $B_{2n}$ for some positive integer $n$. This result can be also retrieved by the classical Alexander theorem. Simply put $L$ in standard closure form of a braid $β \in B_n$ and then drag the out-most upwards strand between the first and second strand of $β$. Then take the newly created out-most upwards strand and drag it between the second and third strand of $β$. Doing this inductively shall yield a $β_0 \in B_{2n}$ so that $L$ is isotopic to $a^*_nβ_0a^{*-}_n$ (cf. \cite{cavicchioli2023passing}).

In \cite{birman1976stable}, Birman used Hilden's generators of the Hilden group and proved the following analogue of the Markov theorem:

\begin{theorem}[Birman]
\label{th:Birman}
Given two braids $β \in B_{2n}$ and $β' \in B_{2m}$ with isotopic plat closures, then there exists a sequence of braids
$β = β_0 \to β_1 \to β_2 \to \dots \to β_N = β'$
such that the plat closure of each $β_i$ is equal to that of $β$ or $β'$,
and such that each move $β_i \to β_{i+1}$ is either a double co-set move or a stabilization move. More precisely the moves are: 

\begin{enumerate}
    \item $\quad β \leftrightarrow h_1βh_2 \quad \textrm{where} \quad h_1,h_2 \in H_{2n} \quad \textrm{(Double coset move in $B_{2n}$)}$
    \item $\quad β \leftrightarrow βσ_{2n}^{\pm} \in B_{2n+2} \quad \textrm{(Stabilization move)}$
\end{enumerate}

\end{theorem}

\begin{remark}
 For an algorithmic transition from standard closure to plat closure and vice versa for the cases of classical links, links in a handlebody and links in thickened surfaces cf.~\cite{cavicchioli2023passing}.    
\end{remark}


\section{Framed Hilden Groups and Pure Framed Hilden Groups}
\label{sec:framedHilden_plat}

Our goal is to define the framed Hilden groups as subgroups of the framed braid groups, give a presentation for this type of groups analogous to Tawn's presentation, and proceed with defining and presenting the pure framed Hilden groups.

\subsection{The framed Hilden group}

Let $\{D_i\}_{i=1}^n$ be a collection of pairwise disjoint closed discs properly embedded in the interior of the disc $D^2$ and let $x_i \in \partial D_i$ for $i=1,\ldots,n$. Let $M = D^2- \cup_{i=1}^n intD_i$. Let $f:M \to M$ be a homeomorphism so that $f(x)=x$ for all $x \in \partial D^2$ and that $f$ fixes the set $Q = \{x_1,\ldots,x_n\}$ set-wise. Because $f$ is the identity on the boundary of $D^2$ then $f$ fixes set-wise $\cup_{i=1}^n \partial D_i$ so that $f(\partial D_i)= \partial D_j$ if and only if $f(x_i)=x_j$. Since $\partial D_i$ is homeomorphic to $S^1$ then we can radially expand $f$ to the interiors of $D_i$. Abusing notation we shall also denote this new homeomorphism $f$. Now, $f$ induces a permutation to the radii $r_i$ connecting the center $c_i$ of $D_i$ with $x_i$. Our new $f$ is now an orientation preserving homeomorphism of $D^2$ that is the identity on $\partial D^2$. By a theorem of Alexander, there is an isotopy $\{h_t\}_{t \in [0,1]}$ so that $h_t:D^2 \to D^2$ is a homeomorphism keeping the boundary fixed point-wise for every $t \in [0,1]$, $h_1=id$  and $h_0=f$. The set $\cup_{i=1}^n\{(h_t(r_i),t) \, | \, t\in [0,1]\}$ is a geometric framed braid embedded in the cylinder $D^2 \times [0,1]$. It is easy to see that the mapping class group MCG$(M,Q)$ is isomorphic to the framed braid group $RB_n$ on $n$ ribbons (see \cite{prasolov1997knots}, page 63). Note that, also in the framed case, our homeomorphisms are the identity on the boundary of $D^2$ and not on the boundary of $M$.

\begin{figure}[htp]
    \centering
    \includegraphics[width=10cm]{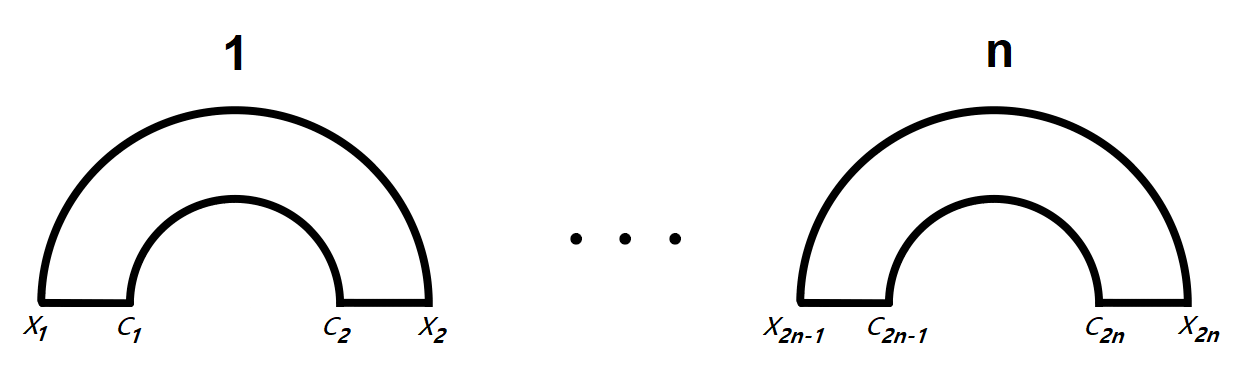}
    \caption{The element $A^*_n$.}
    \label{fig:framed_A*}
\end{figure} 

 Let us consider now the diagram in Fig.~\ref{fig:Hilden1} as a diagram of $n$ contracted ribbons. We denote $A^*_n$ to be the diagram in Fig. \ref{fig:framed_A*}. Recall that a ribbon is a set parametrized as $γ \times [0,1]$ where $γ$ is an arc and the endpoint intervals of the ribbon are the sets $γ(0) \times [0,1]$ and $γ(1) \times [0,1]$.
 
\begin{definition}
     Let $D^2$ be canonically embedded in $\partial \Bbb{R}^3_+$ and let the endpoint intervals of $A^*_n$ lie entirely in the interior of $D^2$. Then the subgroup of $B_{2n}$ that stabilizes $A^*_n$ shall be called the \textit{Framed Hilden Braid Group on 2n ribbons} and denoted $RH_{2n}$. Equivalently, $RH_{2n}$ consists of the equivalence classes of self-homeomorphisms of $(M,Q)$ where $M = D^2- \cup_{i=1}^{2n} intD_i$ and $Q$ is the set of points $\{x_1,\ldots , x_{2n} \}$ of $A^*_n$,  that can be extended to self-homeomorphisms of $(\Bbb{R}^3_+,A^*_n)$.
\end{definition}

\begin{proposition}
\label{FrHilGr}
   The Framed Hilden group $RH_{2n}$ is isomorphic to the semi-direct product $\mathbb{Z}^n \rtimes H_{2n}$.
\end{proposition}

\begin{proof}
    Let $π:RB_{2n} \to B_{2n}$ be the natural projection that forgets the framing. Define elements in $RB_{2n}$ so that

\begin{align*}
    p_i &= σ_{2i}σ_{2i-1}σ_{2i+1}^{-1}σ_{2i}^{-1} \quad \textrm{for} \quad 1 \le i \le n-1\\
    s_j &= σ_{2j}σ_{2j-1}σ_{2j+1}σ_{2j} \quad \textrm{for} \quad 1 \le j \le n-1\\
    θ_k &= t_{2k-1}σ_{2k-1} \quad \quad \quad \quad \textrm{for} \quad 1 \le k \le n\\
    ω_λ &= t_{2λ-1}t_{2λ}^{-1} \quad \quad \quad \quad \quad \textrm{for} \quad 1 \le λ \le n
\end{align*} where $σ_i,t_j$ are the generators of $RB_{2n}$ (see Figs. \ref{fig:Hilden_generators} and \ref{fig:new_generators}).\\
\\
\indent Let $Κ_{2n} = <p_1,\ldots,p_{n-1},s_1,\ldots,s_{n-1},θ_1,\ldots,θ_n>$. The generators of $Κ_{2n}$ satisfy Tawn's relations in the presentation of $H_{2n}$ so there is a well defined group homomorphism $φ:H_{2n} \to Κ_{2n}$, that is mutually inverse with the restriction of $π$ to $Κ_{2n}$. Therefore $Κ_{2n} \cong H_{2n}$ and we can think of $Η_{2n}$ as a subgroup of $RB_{2n}$ identified with $Κ_{2n}$. We now take the restriction of $π$ to $RH_{2n}$. It's easy to see that $π(RH_{2n})$ is in $H_{2n}$. Indeed, let $b$ be a framed braid that stabilizes $A^*_n$, i.e. the diagram $A^*_nb=A^*_n$. Then contracting the diagrams $A^*_nb, A^*_n$ to the center-lines of the ribbons we have $a^*_nπ(b)=a^*_n$, therefore $π(b) \in H_{2n}$. It is clear that $π|_{RH_{2n}}$ is onto. We only need to find the kernel of $π|_{RH_{2n}}$. Let $b \in RH_{2n}$, so that $π(b)=1$. Because $b$ is a framed braid then $b=t_1^{λ_1}\ldots t_{2n}^{λ_{2n}}β$, with $β \in B_{2n}$. Because $π(b)=β=1$, this means that $b=t_1^{λ_1} \ldots t_{2n}^{λ_{2n}}$. However, since $b \in RH_{2n}$ this means that $A^*_nt_1^{λ_1} \ldots t_{2n}^{λ_{2n}}=A^*_n$ and thus $λ_{2i-1}+λ_{2i}=0$ for all $i=1,\ldots,n$. This proves that the kernel of $π|_{RH_{2n}}$ is $<ω_1,\ldots,ω_{n}> \cong \mathbb{Z}^n$.
 
    Take an element $x \in \textrm{ker}\, π|_{RH_{2n}} \cap Κ_{2n}$, then $x=t_1^{λ_1}t_2^{-λ_1}\ldots t_{2n-1}^{λ_n}t_{2n}^{-λ_n} \in Κ_{2n}$. However, since $π|_{K_{2n}}$ is an isomorphism and $π(x)=1$, then $x=1$. Let now $x \in RH_{2n}$, then $x \, φ \circ π(x)^{-1} \in \textrm{ker}π|_{RH_{2n}}$, so that $x = t_1^{λ_1}t_2^{-λ_1}\ldots t_{2n-1}^{λ_n}t_{2n}^{-λ_n}a$, where $a \in K_{2n}$. This proves that $RH_{2n} = \textrm{ker}\, π|_{RH_{2n}} \rtimes K_{2n}  \cong \mathbb{Z}^n \rtimes_ψ H_{2n}$, where $ψ:H_{2n} \to \textrm{Aut}(\mathbb{Z}^n)$ induced by the action of the braid generators $σ_i$ on the framing generators $t_j$.
\end{proof}

\begin{figure}[htp]
    \centering
    \includegraphics[width=6cm]{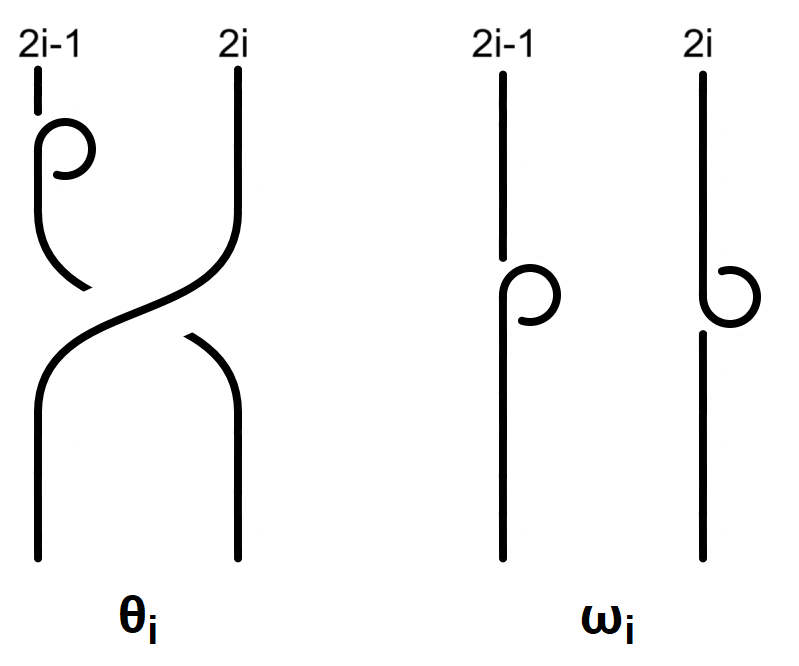}
    \caption{The elements $θ_i$ and $ω_i$.}
    \label{fig:new_generators}
\end{figure} 

\begin{corollary}
    The Framed Hilden group $RH_{2n}$ has a presentation with generators:
    \begin{align*} 
p_i &\quad \textrm{for} \quad 1 \le i \le n-1\\
s_j &\quad \textrm{for} \quad 1 \le j \le n-1\\
θ_k &\quad \textrm{for} \quad 1 \le k \le n\\
ω_λ &\quad \textrm{for} \quad 1 \le λ \le n
\end{align*} and relations (\ref{Hilden relations1}) for $p_i,s_j,θ_k$, together with the following relations:
\begin{align*}
     p_jω_i &= ω_{i+1}p_j\phantom{T}  \quad \textrm{for} \quad i=j \\
     p_jω_i &= ω_{i-1}p_j\phantom{T}  \quad \textrm{for} \quad i=j+1\\
     p_jω_i &= ω_ip_j\phantom{T} \quad \quad \textrm{for} \quad i \neq j,j+1 \\
     s_jω_i &= ω_{i+1}s_j\phantom{T} \quad \textrm{for} \quad i=j \\
     s_jω_i &= ω_{i-1}s_j\phantom{T} \quad \textrm{for} \quad i=j+1 \\
     s_jω_i &= ω_is_j\phantom{T} \quad \quad \textrm{for} \quad i \neq j,j+1 \\
     θ_jω_i &= ω_i^{-1}θ_j \quad \quad \textrm{for} \quad i=j \\
     θ_jω_i &= ω_iθ_j\phantom{T} \quad \quad \textrm{for} \quad i \neq j \\
     ω_iω_j &= ω_jω_i\phantom{T} \quad \quad \textrm{for} \quad 1 \le i,j \le n 
\end{align*}

\end{corollary}

\begin{proof}
    This is an immediate consequence of Proposition \ref{FrHilGr} and the presentation of the semi-direct product of two groups.
\end{proof}

\begin{figure}[htp]
    \centering
    \includegraphics[width=8cm]{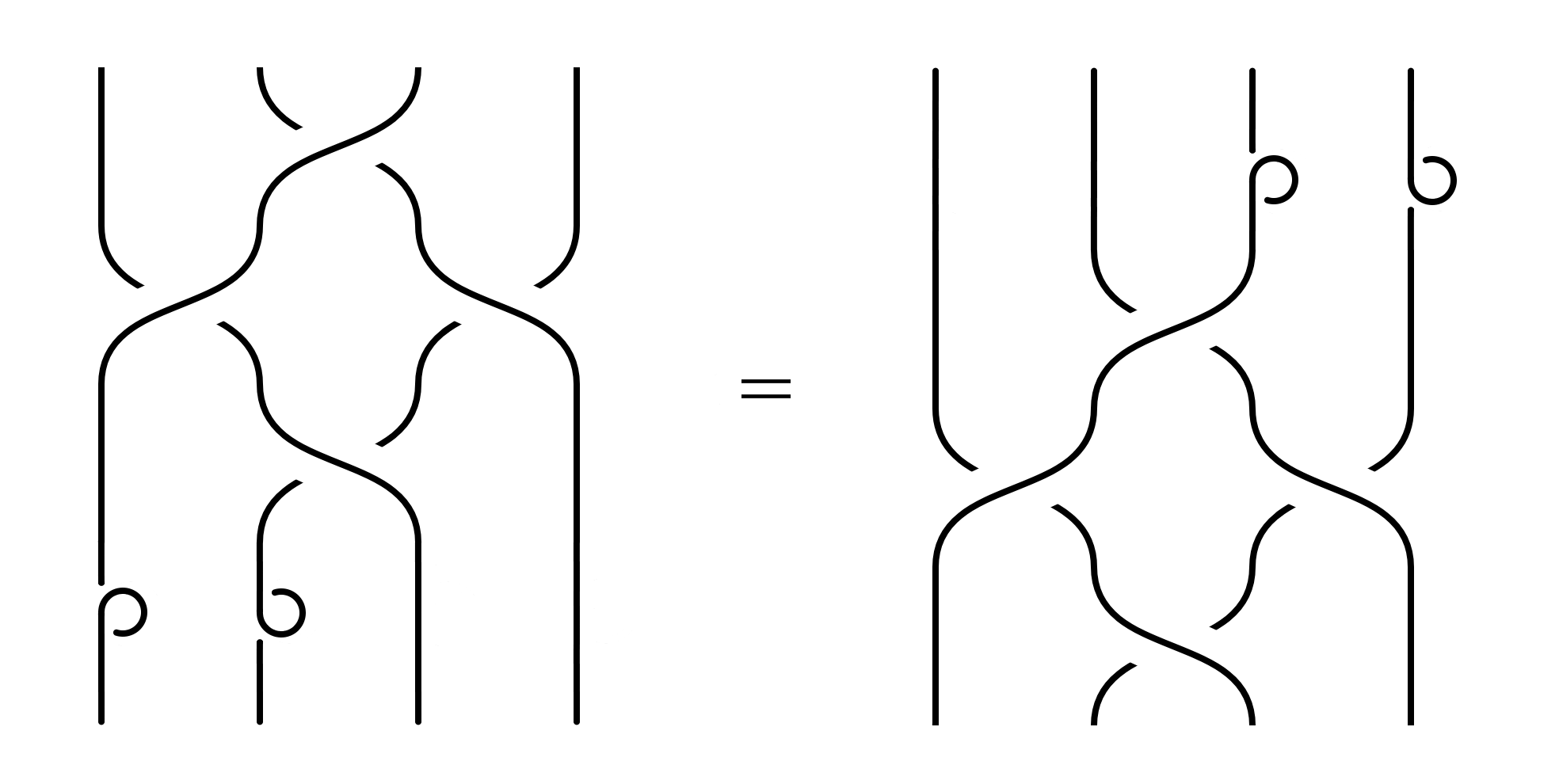}
    \caption{Relation $p_iω_i = ω_{i+1}p_i$.}
    \label{fig:1st_rel}
\end{figure} 

\begin{figure}[htp]
    \centering
    \includegraphics[width=8cm]{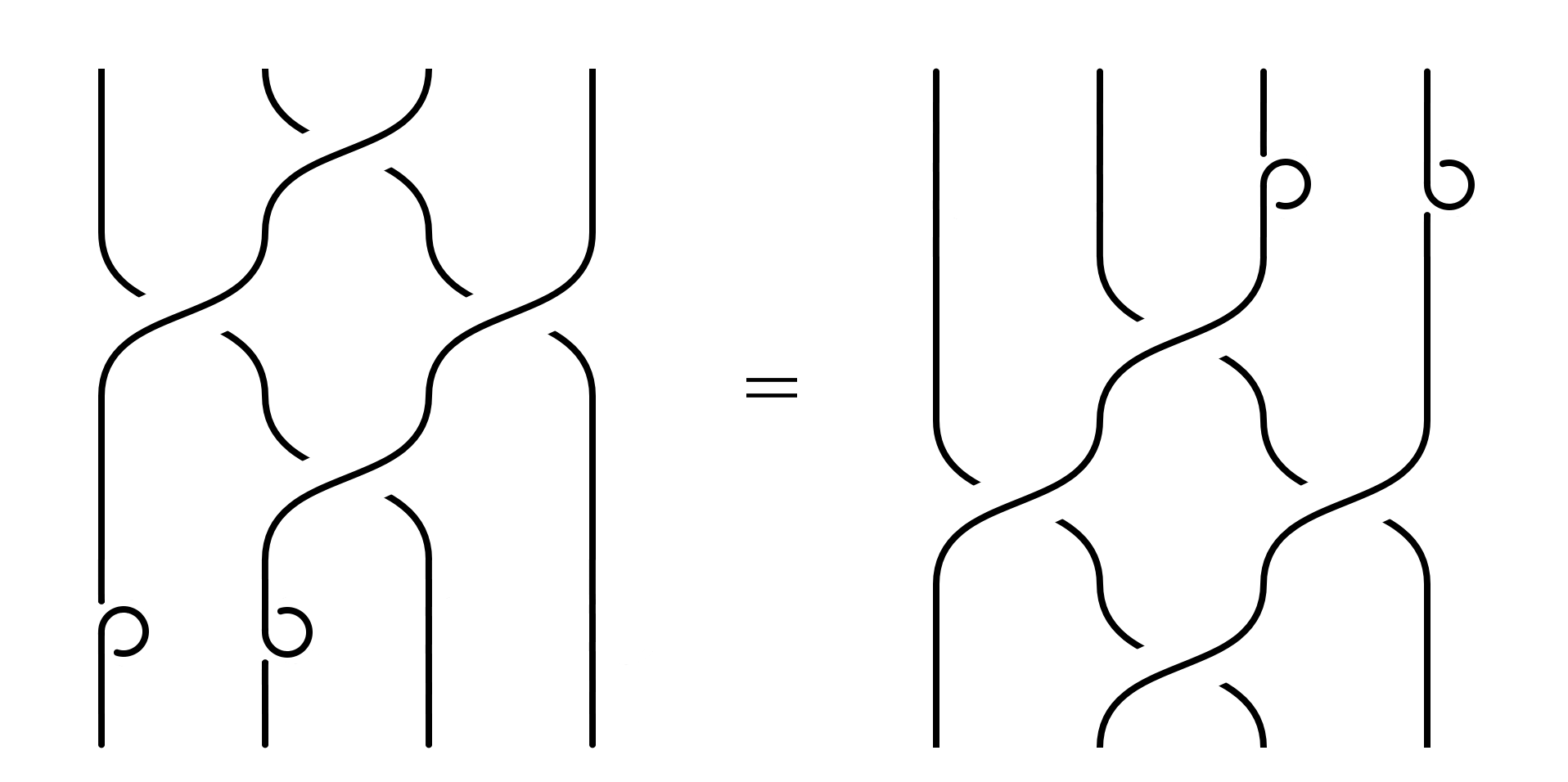}
    \caption{Relation $s_iω_i = ω_{i+1}s_i$.}
    \label{fig:2nd_rel}
\end{figure} 

\begin{figure}[htp]
    \centering
    \includegraphics[width=4.5cm]{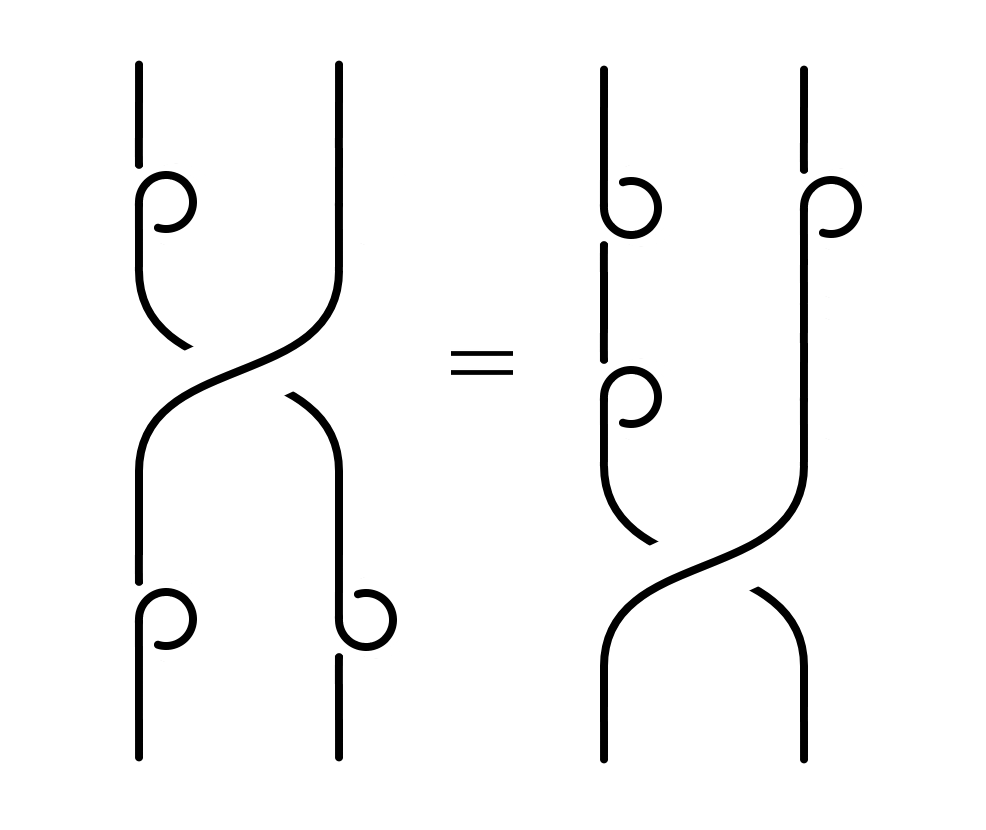}
    \caption{Relation $θ_iω_i = ω_i^{-1}θ_i$.}
    \label{fig:3rd_rel}
\end{figure} 

\begin{figure}[htp]
    \centering
    \includegraphics[width=8cm]{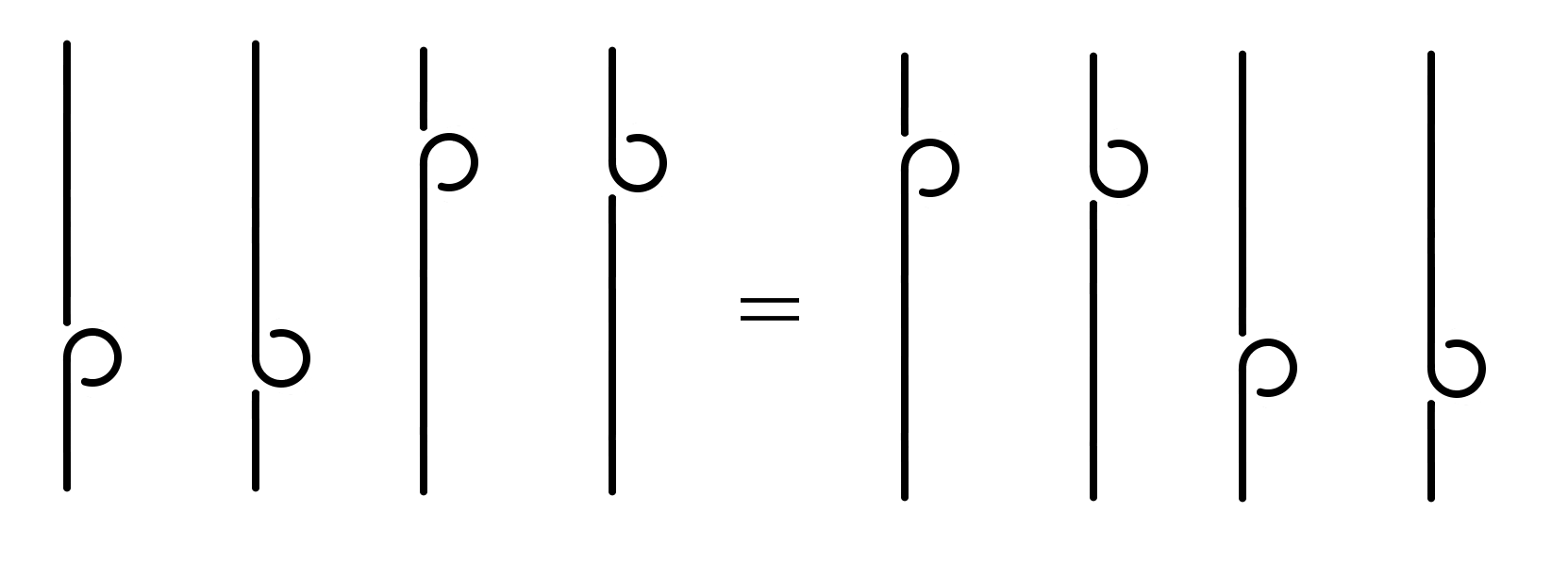}
    \caption{Relation $ω_{i+1}ω_i = ω_iω_{i+1}$.}
    \label{fig:4th_rel}
\end{figure}

\noindent For a depiction of some of the framed Hilden group relations see Figs. \ref{fig:1st_rel}, \ref{fig:2nd_rel}, \ref{fig:3rd_rel} and \ref{fig:4th_rel}.

\subsection{The pure framed Hilden group}

In \cite{tawn2009presentation} Tawn has also provided a presentation for the pure Hilden group $PH_{2n} = H_{2n} \cap P_{2n}$, where $P_{2n}$ is the pure braid group on $2n$ strands. Namely, Tawn showed that the pure Hilden group is generated by the elements depicted in Fig.~\ref{purediagram} where $1 \le i < j \le n$. Using the same method as in the proof of Proposition~\ref{FrHilGr} we can show that $PRH_{2n} = RH_{2n} \cap RP_{2n}$, the pure framed Hilden group, is isomorphic to $\mathbb{Z}^n \oplus PH_{2n}$.

\begin{figure}[htp]
    \centering
    \includegraphics[width=13cm]{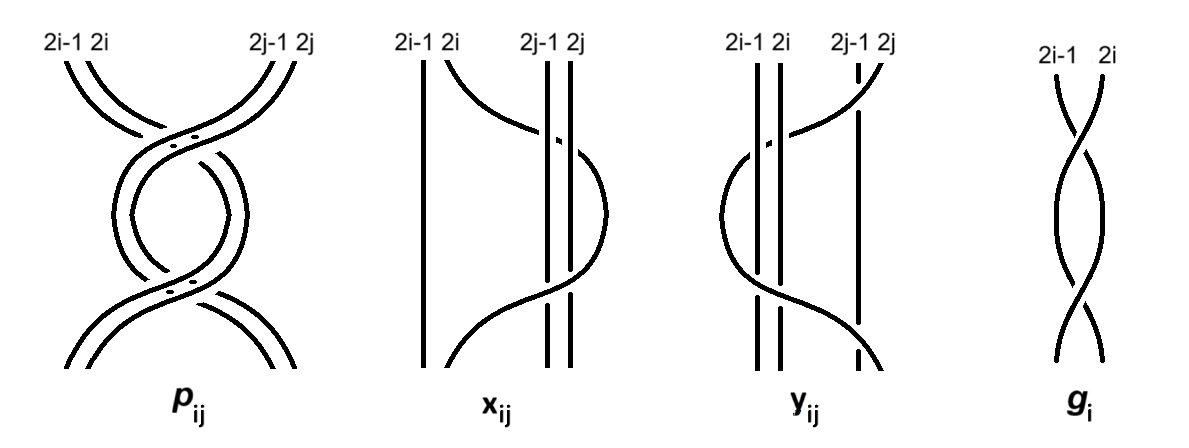}
    \caption{The generators of $PH_{2n}$.}
    \label{purediagram}
\end{figure}

\begin{proposition}
\label{PureFrHilGr}
   The pure framed Hilden group $PRH_{2n}$ is isomorphic to the direct sum $\mathbb{Z}^n \oplus PH_{2n}$.
\end{proposition}

\begin{proof}
    Define elements $p_{ij}, x_{ij}, y_{ij}, \quad \textrm{for} \quad 1 \le i < j \le n$ in $RP_{2n}$, so that they correspond to the diagrams in Fig. \ref{purediagram}. Note that any ribbons between ribbon number $2i$ and ribbon number $2j-1$ are vertical ribbons that pass entirely under the ribbons depicted in the diagram. Define also elements $g_k = t_{2k-1}t_{2k}σ_{2k-1}^2, \quad \textrm{for} \quad 1 \le k \le n $, and take further the elements $ ω_λ = t_{2λ-1}t_{2λ}^{-1},  \quad \textrm{for} \quad 1 \le λ \le n $, as in Proposition \ref{FrHilGr}. These elements satisfy Tawn's relations in \cite{tawn2009presentation} for $PH_{2n}$ (see relations below not involving the $\omega_i$'s), so there is a well-defined isomorphism $φ: PH_{2n} \to Δ_{2n}:= < p_{ij}, x_{ij}, y_{ij}, g_k   \ | \ 1 \le i < j \le n, \  1 \le k \le n>$, with $π|_{PRH_{2n}}$ as its mutual inverse. Then $\textrm{ker}\, π|_{PRH_{2n}} = <ω_1,\ldots, ω_1>$. The same arguments as in the proof of Proposition~\ref{FrHilGr} show that $PRH_{2n} = \textrm{ker}\, π|_{PRH_{2n}} \oplus Δ_{2n} \cong \mathbb{Z}^n \oplus PH_{2n}$ since $Δ_{2n}$ is normal in $PRH_{2n}$.
\end{proof}

\begin{corollary}
    The Pure Framed Hilden group $PRH_{2n}$ has a presentation with generators:\begin{align*}
    p_{ij} &\quad \textrm{for} \quad 1 \le i < j \le n\\
    x_{ij} &\quad \textrm{for} \quad 1 \le i < j \le n\\
    y_{ij} &\quad \textrm{for} \quad 1 \le i < j \le n\\
    g_k  &\quad \textrm{for} \quad 1 \le k \le n\\
    ω_λ  &\quad \textrm{for} \quad 1 \le λ \le n   
\end{align*}    
    where $g_k,ω_λ$ as defined in the proof of Proposition \ref{PureFrHilGr} and $p_{ij},x_{ij},y_{ij}$ as in Fig. \ref{purediagram}  and relations: \begin{align*}
        p_{ij}g_k &= g_kp_{ij}\\
        g_ig_j &= g_jg_i\\
        x_{ij}g_k &= g_kx_{ij} \quad \quad \quad \quad i<j \quad k \neq i\\
        y_{ij}g_k &= g_ky_{ij} \quad \quad \quad \quad i<j \quad k \neq j
    \end{align*} \begin{align*}
        α_{ij}β_{kl} &= β_{kl}α_{ij} \quad \quad \quad \quad \quad \quad α,β \in \{p,x,y\} \quad \textrm{and} \quad (i,j,k,l) \quad \textrm{cyclically ordered}\\
        α_{ij}β_{ik}γ_{jk} &= β_{ik}γ_{jk}α_{ij} \quad \quad \quad \quad (i,j,k) \quad \textrm{cyclically ordered} \quad \textrm{and} \quad (α,β,γ) \quad \textrm{as in Table \ref{Table}}\\
    α_{ik}p_{jk}β_{jl}p_{jk}^{]-1} &= p_{jk}β_{jl}p_{jk}^{]-1}α_{ik} \quad \quad \quad α,β \in \{p,x,y\} \quad \textrm{and} \quad (i,j,k,l) \quad \textrm{cyclically ordered}
    \end{align*} \begin{align*}
    x_{ij}p_{ij}g_i &= p_{ij}g_ix_{ij} \quad \quad \quad \quad i<j\\
    y_{ij}p_{ij}g_i &= p_{ij}g_iy_{ij} \quad \quad \quad \quad i<j\\
    ω_iω_j &= ω_jω_i\phantom{t} \quad \quad \textrm{for} \quad 1 \le i,j \le n\\
    ω_ip_{kj} &= p_{kj}ω_i \quad \quad \textrm{for} \quad 1 \le i,j,k \le n\\
    ω_ix_{kj} &= x_{kj}ω_i \quad \quad \textrm{for} \quad 1 \le i,j,k \le n\\
    ω_iy_{kj} &= y_{kj}ω_i \quad \quad \textrm{for} \quad 1 \le i,j,k \le n
    \end{align*}
    
\end{corollary}

\begin{table}[!h]
\begin{center}
\begin{tabular}{ m{5em} | m{5.5cm} } 
  $i < j < k$ & $(p,p,p) \quad (p,y,y) \quad (x,p,p) \quad (x,x,p)$  \\
              & $(x,y,y) \quad (y,p,p) \quad (y,p,x) \quad (y,y,y)$  \\
  \hline
  $j < k < i$  & $(p,p,p) \quad (p,x,y) \quad (x,p,p) \quad (x,p,x)$  \\
               & $(x,x,y) \quad (y,p,p) \quad (y,x,y) \quad (y,y,p)$  \\ 
  \hline
  $k < i < j$  & $(p,p,p) \quad (p,x,x) \quad (x,p,p) \quad (x,x,x)$  \\
               & $(x,y,p) \quad (y,p,p) \quad (y,p,y) \quad (y,x,x)$  \\              
\end{tabular}
\caption{\label{demo-table}The values of $(α,β,γ)$ from \cite{tawn2009presentation}.}
\label{Table}
\end{center}
\end{table}

\begin{proof}
    Immediate consequence of Proposition \ref{PureFrHilGr} and the presentation of the direct sum of two groups.
\end{proof}


\section{Plat Closure of Framed Braids and Framed Plat Equivalence}
\label{sec:framed_birman}

The aim of this section is to give an exact analogue of the Birman plat equivalence theorem for the plat closure of framed braids. Let $A^{*-}_n$ be the mirror image of the element $A^{*}_n$ with respect to the $xy$-plane. Notice that $A^*_n$ and $A^{*-}_n$ are the $(0,2n)$ and $(2n,0)$ untwisted ribbon tangles whose contracted tangle diagrams correspond to $a^*_n$ and $a^{*-}_n$ respectively. As in the case of the framed Markov theorem, we shall need to modify the stabilization move in order to respect the framing, since some elements of the Hilden groups  do not respect the framing. So, we need to make use of the framed Hilden groups $RH_{2n}$.

\begin{definition}
    The \textit{plat closure of a framed braid} $b \in RB_{2n}$ is defined to be the closed diagram $A^*_nbA^{*-}_n$.
\end{definition}

Just like in the case of classical links, given a framed link $L$ and a diagram of $L$ in standard closure form of a framed braid, we can recover a framed plat representation of $L$ simply by dragging the untwisted closure ribbon arcs, as described above Theorem \ref{th:Birman}. We then have the following:

\begin{proposition}
    Every framed link can be regarded as the plat closure of some framed braid with even number of ribbons.
\end{proposition}

 The following framed version of the Birman theorem tells us when two framed braids with even number of ribbons produce isotopic links when plat closed.

\begin{theorem}[Framed Birman]
Given two framed braids $b \in RB_{2n}$ and $b' \in RB_{2m}$ with isotopic framed plat closures, then there exists a sequence of framed braids
$b = b_0 \to b_1 \to b_2 \to \ldots \to b_N = b'$
such that the framed plat closure of each $b_i$ is equal to that of $b$ or $b'$,
and such that each move $b_i \to b_{i+1}$ is either a double coset move or a
framed stabilization move. More precisely the moves are: \begin{align*}
    b &\leftrightarrow h_1bh_2 \quad \textrm{where} \quad h_1,h_2 \in RH_{2n} \quad \textrm{(Double co-set move in $RB_{2n}$)}\\
    b &\leftrightarrow bt_{2n}^{\mp}σ_{2n}^{\pm} \in RB_{2n+2} \quad \quad \quad \quad \quad \quad \textrm{(Framed Stabilization move)}   
\end{align*}

\end{theorem}

\begin{proof}
    We shall use the same idea that we used for the Framed $L$-move theorem. For the `if' part, it is clear that the plat closure of a framed braid after a double coset move or a framed stabilization move (with the framing adjustment) is isotopic to the plat closure of the framed braid before the move takes place.\\    
    \indent For the `only if' part, if two framed braids $b \in RB_{2n}$, \ $b'\in RB_{2m}$ have isotopic plat closures, then their projections in $B_{2n}$ and $B_{2m}$ by the mapping $π$ that forgets the framing, gives us that $π(b) = β \in B_{2n},π(b')=β'\in B_{2m}$, have isotopic plat closures and therefore are connected by a sequence of double co-set and stabilization moves, according to Theorem~\ref{th:Birman}. If a double co-set move takes place in the classical braid level, namely $π(b) \leftrightarrow h_1π(b)h_2$, then we perform the framed double co-set move $b \leftrightarrow φ(h_1)bφ(h_2)$, in the framed braid level, where $φ$ is the isomorphism $φ:H_{2n} \to Κ_{2n}$ in the proof of proposition \ref{FrHilGr}. If a stabilization move takes place in the classical braid level, namely $π(b) \leftrightarrow π(b)σ_{2n}^{\pm} $, then we perform the framed stabilization move $b \leftrightarrow bt_{2n}^{\mp}σ_{2n}^{\pm} $ in the framed braid level. Then, if $b=t_1^{λ_1}\ldots t_{2n}^{λ_{2n}}β$ and $b'=t_1^{δ_1}\ldots t_{2m}^{δ_{2m}}β'$, we shall have a finite sequence of double co-set and stabilization moves from the framed braid $b$ to $b''$ where $b''=t_1^{κ_1}\ldots  t_{2m}^{κ_{2m}}β'$. Since the plat closure of $b'$ and $b''$ represents isotopic framed links this will mean that $\sum_{i=0}^{2m}δ_i = \sum_{i=0}^{2m}κ_i$. We shall now modify the framing accordingly using double co-set moves of the form $Λ_{2m}b''Λ_{2m}$, where $Λ_{2m}=<ω_1,\ldots ,ω_{m}>$. More precisely, let $x = ω_1^{x_1}\ldots ω_m^{x_m}$ and $y = ω_1^{y_1}\ldots ω_m^{y_m}$, where $ω_i=t_{2i-1}t_{2i}^{-1}$. We need to solve the system of equations derived from $xb''y=b'$. But as we have seen (recall the proof of Theorem~\ref{framed L-Markov}) such a system always has infinitely many integer solutions.
\end{proof}

\nocite{*}
\bibliographystyle{plain}
\bibliography{refs}

School of Applied Mathematical and Physical Sciences, National Technical University of Athens, Zografou campus, GR-15780 Athens, Greece.\\
\indent akokkinakis@mail.ntua.gr

\end{document}